\documentclass[10pt,conference,letterpaper]{IEEEtran}
\usepackage{booktabs} 
\usepackage[font=small]{caption}
\usepackage[labelformat=simple]{subcaption}
\usepackage{amssymb, graphicx}
\usepackage[ruled,vlined,linesnumbered]{algorithm2e}
\usepackage{algorithmic}
\SetKwComment{Comment}{$\triangleright$\ }{}
\usepackage{indentfirst}
\usepackage{mdwlist}
\usepackage[mathscr]{eucal}
\usepackage{changepage}
\usepackage{color}
\usepackage{url}
\usepackage{hyphenat}
\usepackage[font=small]{caption}
\usepackage{makecell}
\usepackage{mathtools}
\usepackage{xcolor}
\usepackage{textcomp}
\usepackage[bookmarks=false]{hyperref}
\usepackage{xspace}
\usepackage{cleveref}
\usepackage{footnote}
\usepackage{multirow}
\usepackage{soul}
\usepackage{times}
\usepackage{amssymb}

\usepackage{etoolbox}

\newcommand{\zerodisplayskips}{%
	\setlength{\abovedisplayskip}{1pt}%
	\setlength{\belowdisplayskip}{1pt}%
	\setlength{\abovedisplayshortskip}{0pt}%
	\setlength{\belowdisplayshortskip}{0pt}}
\appto{\small}{\zerodisplayskips}

\allowdisplaybreaks

\newcommand{\wopt}{\textsc{Tucker-wOpt}\xspace}
\newcommand{\SHOT}{\textsc{$\text{S-HOT}_{\text{scan}}$}\xspace}
\newcommand{\CSF}{\textsc{Tucker-CSF}\xspace}
\newcommand{\method}{\textsc{P-Tucker}\xspace}
\newcommand{\MOP}{\textsc{P-Tucker}\xspace}
\newcommand{\TOP}{\textsc{P-Tucker-Cache}\xspace}
\newcommand{\APP}{\textsc{P-Tucker-Approx}\xspace}
\newcommand{\omp}{\textsc{OpenMP}\xspace}
\newcommand{\arma}{\textsc{Armadillo}\xspace}

\newtheorem{definition}{Definition}
\newtheorem{proof}{Proof}

\newtheorem{theorem}{Theorem}

\newcommand{\T}[1]{\boldsymbol{\mathscr{#1}}}   
\newcommand{\tensor}[1]{\boldsymbol{\mathscr{#1}}}   
\newcommand{\mat}[1]{\mathbf{#1}}
\newcommand{\vect}[1]{\mathbf{#1}}
\newcommand{\argmin}[1]{\underset{#1}{\operatorname{arg}\,\operatorname{min}}\;}

\begin{document}
	\title{Scalable Tucker Factorization for \\ Sparse Tensors - Algorithms and Discoveries}

\author{%
	{Sejoon Oh{\small $^{\ast}$}, Namyong Park{\small $^{\dagger}$}, Lee Sael{\small $^{\ast}$}, U Kang{\small $^{\ast}$} }%
	\vspace{1.6mm}\\
	\fontsize{10}{10}\selectfont\itshape
	$^{\ast}$ Seoul National University, Korea\\
	$^{\dagger}$ Carnegie Mellon University, USA\\
	\fontsize{9}{9}\selectfont\ttfamily\upshape
	%
	$^{\ast}$\,ohhenrie@snu.ac.kr
	$^{\dagger}$\,namyongp@cs.cmu.edu
	$^{\ast}$\,saellee@gmail.com
	$^{\ast}$\,ukang@snu.ac.kr%
}
%
%
%
%
%
%
	
	\maketitle
	
	\begin{abstract}
		\label{sec:abstract}
		\noindent
Given sparse multi-dimensional data (e.g., (user, movie, time; rating) for movie recommendations), how can we discover latent concepts/relations and predict missing values?
Tucker factorization has been widely used to solve such problems with multi-dimensional data, which are modeled as tensors.
However, most Tucker factorization algorithms regard and estimate missing entries as zeros, which triggers a highly inaccurate decomposition. Moreover, few methods focusing on an accuracy exhibit limited scalability since they require huge memory and heavy computational costs while updating factor matrices.

In this paper, we propose \method, a scalable Tucker factorization method for sparse tensors.
\method performs alternating least squares with a row-wise update rule in a fully parallel way, which significantly reduces memory requirements for updating factor matrices.
Furthermore, we offer two variants of \method: a caching algorithm \TOP and an approximation algorithm \APP, both of which accelerate the update process.
Experimental results show that \method exhibits 1.7-14.1$\times$ speed-up and 1.4-4.8$\times$ less error compared to the state-of-the-art. In addition, \method scales near linearly with the number of observable entries in a tensor and number of threads. 
Thanks to \method, we successfully discover hidden concepts and relations in a large-scale real-world tensor, while existing methods cannot reveal latent features due to their limited scalability or low accuracy. 
	\end{abstract}

%
%

	\section{Introduction}
	\label{sec:intro}


Given a large-scale sparse tensor, how can we discover latent concepts/relations and predict missing entries?
How can we design a time and memory efficient algorithm for analyzing a given tensor?
Various real-world data can be modeled as tensors or multi-dimensional arrays (e.g., (user, movie, time; rating) for movie recommendations).
Many real-world tensors are sparse and partially observable, i.e.,
composed of a vast number of missing entries and a relatively small number of observable entries.
Examples of such data include item ratings~\cite{Yahoo}, social network~\cite{DBLP:journals/tkdd/DunlavyKA11}, and web search logs~\cite{Websearch} where most entries are missing.
Tensor factorization has been used effectively for analyzing tensors~\cite{human, Anomaly, zheng2010flickr,journals/debu/PapalexakisKFSH13,Sael201582,conf/cikm/ParkJLK16,conf/icde/JeonJSK16}.
Among tensor factorization methods~\cite{kolda2009tensor}, Tucker factorization has received much interest since it is a generalized form of other factorization methods like CANDECOMP/PARAFAC (CP) decomposition, and it allows us to examine not only latent factors but also  relations hidden in tensors.

While many algorithms have been developed for Tucker factorization~\cite{tucker1966some, DBLP:journals/siammax/LathauwerMV00a, MET, DBLP:journals/vldb/JeonPFSK16},
most methods produce highly inaccurate factorizations since they assume and predict missing entries as zeros, and the values of whose missing entries are unknown.
Moreover, existing methods focusing only on observed entries exhibit limited scalability since they exploit tensor operations and singular value decomposition (SVD), leading to heavy memory and computational requirements.
In particular, tensor operations generate huge intermediate data for large-scale tensors,
which is a problem called \textit{intermediate data explosion}~\cite{kang2012gigatensor}.
A few Tucker algorithms~\cite{Oh:2017:SHOT,filipovic2015tucker, kaya, smith2017tucker} have been developed to address the above problems,
but they fail to solve the scalability and accuracy issues at the same time.
In summary, the major challenges for decomposing sparse tensors are
1) how to handle missing entries for an accurate and scalable factorization, and
2) how to avoid intermediate data explosion and high computational costs caused by tensor operations and SVD.

\begin{table}[t!]
	\small
	\caption{
		Scalability summary of our proposed method \method and competitors. A check-mark of a method indicates that the algorithm is scalable with a particular aspect.
			\method is the only method scalable with all aspects of tensor scale, factorization speed, memory requirement, and accuracy of decomposition; on the other hand, competitors have limited scalability for some aspects.
	}
	\centering
		\begin{tabular}{c | c c c c c }
			\toprule
			\textbf{Method} & \textbf{Scale} & \textbf{Speed} & \textbf{Memory} & \textbf{Accuracy} \\
			\midrule
			\wopt~\cite{filipovic2015tucker} &  &  &  & \checkmark \\
			\CSF~\cite{smith2017tucker}	& \checkmark & \checkmark &  &  \\
			\SHOT~\cite{Oh:2017:SHOT}	& \checkmark & \checkmark & \checkmark & \\
			\midrule
			\textbf{\method}		& \checkmark & \checkmark & \checkmark & \checkmark \\
			\bottomrule
		\end{tabular}
	\label{tab:comparators}
\end{table}

In this paper, we propose \method, a scalable Tucker factorization method for sparse tensors.
\method performs alternating least squares (ALS) with a row-wise update rule, which focuses only on observed entries of a tensor. The row-wise updates considerably reduce the amount of memory required for updating factor matrices, enabling \method to avoid the \textit{intermediate data explosion} problem.
Besides, to speed up the update procedure, we provide its time-optimized versions: a caching method \TOP and an approximation method \APP.
\method fully employs multi-core parallelism by carefully allocating rows of a factor matrix to each thread considering independence and fairness.
Table~\ref{tab:comparators} summarizes a comparison of \method and competitors with regard to various aspects.

Our main contributions are the following:
\begin{itemize}
\item \textbf{Algorithm.}
We propose \method, a scalable Tucker factorization method for sparse tensors. The key ideas of \method include 1) row-wise updates of factor matrices, 2) careful parallelization, and 3) time-optimized variants: \TOP and \APP.

\item \textbf{Theory.}
	We theoretically derive a row-wise update rule of factor matrices, and prove the correctness and convergence of it.
Moreover, we analyze the time and memory complexities of \method and other methods, as summarized in Table~\ref{table::analysis}.
\item \textbf{Performance.} \method provides the best performance across all aspects: tensor scale, factorization speed,  memory requirement, and accuracy of decomposition. Experimental results demonstrate that \method achieves \textbf{1.7-14.1}$\times$ speed-up with \textbf{1.4-4.8$\times$} less error for large-scale tensors, as summarized in Figures~\ref{fig:scalability}, \ref{fig:real_world_time}, and \ref{fig:real_world_accuracy}.
\end{itemize}


The code of \method and datasets used in this paper are available at \textbf{\url{https://datalab.snu.ac.kr/ptucker/}} for reproducibility.
The rest of this paper is organized as follows.
Section~\ref{sec:preliminary} explains preliminaries on a tensor, its operations, and its factorization methods.
Section~\ref{sec:proposed_method} describes our proposed method \method.
Section~\ref{sec:experiment} presents experimental results of \method and other methods.
Section~\ref{sec:discovery} describes our discovery results on the MovieLens dataset.
After introducing related works in Section~\ref{sec:related_work},
we conclude in Section~\ref{sec:conclusion}.

	\section{Preliminaries}
	\label{sec:preliminary}

We describe the preliminaries of a tensor in Section~\ref{sec:prelim:tensor}, its operations in Section~\ref{sec:prelim:operation}, and its factorization methods in Section~\ref{sec:prelim:tf}. Notations are summarized in Table~\ref{tab:Symbols}.

\vspace{-2mm}
\subsection{Tensor}
\label{sec:prelim:tensor}

Tensors, or multi-dimensional arrays, are a generalization of vectors
($1$-order tensors) and matrices ($2$-order tensors) to higher orders.
As a matrix has rows and columns, an $N$-order tensor has $N$ modes;
their lengths (also called dimensionalities) are denoted by $I_{1}$ through $I_{N}$, respectively.
We denote tensors by boldface Euler script letters (e.g., $\tensor{X}$),
matrices by boldface capitals (e.g., $\mat{A}$), and vectors by boldface lowercases (e.g., $\vect{a}$).
An entry of a tensor is denoted by the symbolic name of the tensor with its indices in subscript.
For example, $a_{i_{1}j_{1}}$ indicates the $(i_{1},j_{1})$th entry of $\mat{A}$, and
$\T{X}_{(i_{1},...,i_{N})}$ denotes the $(i_{1},...,i_{N})$th entry of $\tensor{X}$. The $i_{1}$th row of $\mat{A}$ is denoted by $\vect{a}_{i_{1}:}$, and the $i_{2}$th column of $\mat{A}$ is denoted by $\vect{a}_{:i_{2}}$.

\begin{table}[t!]
	\small
	\centering
	\vspace{-1mm}
	\caption{Table of symbols.}
	\begin{tabular}{c | l}
		\toprule
		\textbf{Symbol} & \textbf{Definition} \\
		\midrule
		$\tensor{X}$ & input tensor $(\in \mathbb{R}^{I_{1} \times ... \times I_{N}})$\\
		$\tensor{G}$ & core tensor $(\in \mathbb{R}^{J_{1}\times ... \times J_{N}})$\\
		$N$ & order of $\tensor{X}$\\
		$I_{n},J_n$ & dimensionality of the $n$th mode of $\tensor{X}$ and $\T{G}$\\
		$\mat{A}^{(n)}$ & $n$th factor matrix $(\in \mathbb{R}^{I_{n} \times J_{n}})$ \\
		$a^{(n)}_{i_{n}j_{n}}$ & $(i_{n},j_{n})$th entry of $\mat{A}^{(n)}$\\
		$\Omega$ & set of observable entries of $\tensor{X}$\\
		$\Omega_{i_n}^{(n)} $ & set of observable entries whose $n$th mode's index is $i_n$\\
		$|\Omega|,|\tensor{G}|$ & number of observable entries of $\tensor{X}$ and $\tensor{G}$\\
		$\lambda$ & regularization parameter for factor matrices\\
		$\|\tensor{X}\|$ & Frobenius norm of tensor $\T{X}$\\
		$T$ & number of threads\\
		$\alpha$ & an entry $(i_1,...,i_N)$ of input tensor $\T{X}$  \\
		$\beta$ & an entry $(j_1,...,j_N)$ of core tensor  $\T{G}$\\
		$Pres$& cache table $(\in \mathbb{R}^{|\Omega| \times |\T{G}|})$ \\
		$p$ & truncation rate \\
		\bottomrule
	\end{tabular}
	\label{tab:Symbols}
	\vspace{-1mm}
\end{table}

\subsection{Tensor Operations}
\label{sec:prelim:operation}
We review some tensor operations used for Tucker factorization. More tensor operations are summarized in~\cite{kolda2009tensor}.

\begin{definition}[Frobenius Norm]
	Given an N-order tensor $\tensor{X}$ $(\in \mathbb{R}^{I_{1} \times ... \times I_{N}})$, the Frobenius norm $||\tensor{X}||$ of $\tensor{X}$ is given by $	||\tensor{X}|| = \sqrt{\sum_{\forall{(i_1,...,i_N)}\in\tensor{X}}{\tensor{X}^{2}_{(i_1,...,i_N)}}}
$.
\\
\end{definition}

\begin{definition}[Matricization/Unfolding]
	Matricization transforms a tensor into a matrix. The mode-$n$ matricization of a tensor $\T{X} \in \mathbb{R}^{I_1 \times I_2 \times \cdots \times I_N}$ is denoted as $\mathbf{X}_{(n)}$. The mapping from an element $(i_1,...,i_N)$ of $\tensor{X}$ to an element $(i_n,j)$ of $\mathbf{X}_{(n)}$ is given as follows:
	\begin{equation}
	j=1+\sum_{k=1, k \neq n}^{N} \bigg[(i_k-1) \prod_{ m=1 , m \neq n}^{k-1} I_m\bigg].
	\end{equation}
	Note that all indices of a tensor and a matrix begin from 1.
\end{definition}

\begin{definition}[n-Mode Product]
	n-mode product enables multiplications between a tensor and a matrix. The n-mode product of a tensor $\T{X} \in \mathbb{R}^{I_1 \times I_2 \times \cdots \times I_N}$ with a matrix $\mathbf{U} \in \mathbb{R}^{J_n \times I_n}$ is denoted by $\tensor{X}\times_{n}\mathbf{U}$ ($\in \mathbb{R}^{I_1 \times \cdots \times I_{n-1} \times J_n \times I_{n+1} \times \cdots \times I_N}$). Element-wise, we have
	\begin{equation}
	{(\tensor{X}\times_{n}\mathbf{U})}_{i_1 \cdots i_{n-1} j_n i_{n+1} \cdots i_N} = \sum_{i_n=1}^{I_n} (\tensor{X}_{(i_1 i_2 \cdots i_N)}  {u}_{j_n i_n}) .
	\end{equation}
\end{definition}

\vspace{-2mm}

\subsection{Tensor Factorization Methods}
\label{sec:prelim:tf}

\begin{figure}[h!]
	\centering
	\includegraphics[height=3cm, width=6cm]{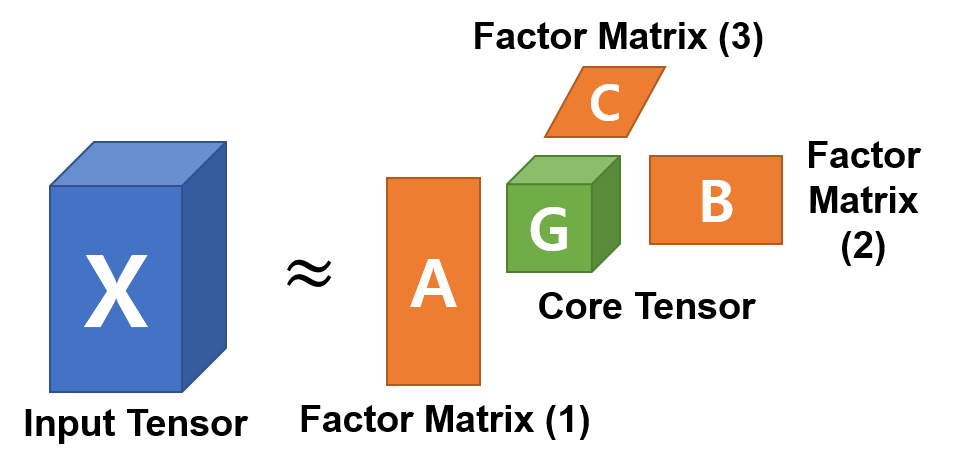}
	\caption{Tucker factorization for a 3-way tensor.}
	\label{fig:3way_tucker}
\end{figure}

Our proposed method \method is based on Tucker factorization, one of the most popular decomposition methods. More details about other factorization algorithms are summarized in Section~\ref{sec:related_work} and \cite{kolda2009tensor}.

\begin{definition}[Tucker Factorization]
	Given an N-order tensor $\tensor{X}$ $(\in \mathbb{R}^{I_{1} \times ... \times I_{N}})$, Tucker factorization
	approximates $\tensor{X}$ by a core tensor $\tensor{G}$ $(\in \mathbb{R}^{J_{1}\times ... \times J_{N}})$ and factor matrices $\{\mathbf{A}^{(n)}\in \mathbb{R}^{I_{n} \times J_n} | n=1...N\}$.
		Figure~\ref{fig:3way_tucker} illustrates a Tucker factorization result for a 3-way tensor.
	Core tensor $\T{G}$ is assumed to be smaller and denser than the input tensor $\T{X}$, and factor matrices $\mathbf{A}^{(n)}$ to be normally orthogonal.
	Regarding interpretations of factorization results,
		each factor matrix $\mathbf{A}^{(n)}$ represents the latent features of the object related to the $n$th mode of $\tensor{X}$, and each element of a core tensor $\tensor{G}$ indicates the weights of the relations composed of columns of factor matrices.
	 Tucker factorization with tensor operations is presented as follows:
	\begin{equation} \label{eq:TF_FULL}
		\min_{\T{G},\mathbf{A}^{(1)},...,\mathbf{A}^{(N)}} || \tensor{X} - \tensor{G}\times_{1}\mathbf{A}^{(1)} \cdots \times_{N}\mathbf{A}^{(N)}||.
	\end{equation}

	Note that the loss function~\eqref{eq:TF_FULL} is calculated by all entries of $\tensor{X}$, and whole missing values of $\tensor{X}$ are regarded as zeros.
	Concurrently, an element-wise expression is given as follows:
	\begin{equation} \label{eq:recon2}
	\tensor{X}_{(i_1,...,i_N)}\approx \sum_{\forall{(j_1,...,j_N)}\in\T{G}}\tensor{G}_{(j_1,...,j_N)}\prod_{n=1}^{N}a^{(n)}_{i_{n}j_n}.
	\end{equation}
		Equation~\eqref{eq:recon2} is used to predict values of missing entries after $\T{G},\mathbf{A}^{(1)},...,\mathbf{A}^{(N)}$ are found.
	We define the reconstruction error of Tucker factorization of $\T{X}$ by the following rule.
	Note that $\Omega$ is the set of observable entries of $\tensor{X}$.
		\begin{equation} \label{eq:reconfull}
		\footnotesize
		\hspace{-5mm}
		 \sqrt{\sum_{\forall{(i_1,...,i_N)}\in\Omega}{\left(\tensor{X}_{(i_1,...,i_N)}-\sum_{\forall{(j_1,...,j_N)}\in\T{G}}\tensor{G}_{(j_1,...,j_N)}\prod_{n=1}^{N}a^{(n)}_{i_{n}j_n}\right)^{2}}}
		\end{equation}

\end{definition}

%
%
%
%

\begin{definition}[Sparse Tucker Factorization] \label{def:partial}
	Given a tensor $\tensor{X}~(\in \mathbb{R}^{I_{1} \times ... \times I_{N}})$
	with observable entries $\Omega$,
	the goal of sparse Tucker factorization of $\tensor{X}$ is to find factor matrices $\mathbf{A}^{(n)}$ $(\in \mathbb{R}^{I_{n} \times J_{n}}, n=1,\cdots,N)$ and a core tensor $\tensor{G}~(\in \mathbb{R}^{J_1 \times...\times J_N})$, which minimize~\eqref{eq:TF_PARTIAL}.
		\begin{multline}
		L(\tensor{G},\mat{A}^{(1)},...,\mat{A}^{(N)}) = \\
		\hspace{-4mm}
		 \sum_{\forall(i_1,...,i_N)\in\Omega}{\left(\tensor{X}_{(i_1,...,i_N)}-\sum_{\forall(j_1,...,j_N)\in\T{G}}\tensor{G}_{(j_1,...,j_N)}\prod_{n=1}^{N}a^{(n)}_{i_{n}j_n}\right)^{2}} \\
		+  \lambda\sum_{n=1}^{N}{{\| \mat{A}^{(n)} \|}^{2}}
		\label{eq:TF_PARTIAL}
		\end{multline}
\end{definition}
	Note that the loss function \eqref{eq:TF_PARTIAL} only depends on observable entries of $\tensor{X}$, and $L_{2}$ regularization is used in \eqref{eq:TF_PARTIAL} to prevent overfitting, which has been generally utilized in machine learning problems \cite{chen2011linear, koren2009matrix, zhou2008large}.

\begin{definition}[Alternating Least Squares]
		\label{def:ALS}
		To minimize the loss functions \eqref{eq:TF_FULL} and \eqref{eq:TF_PARTIAL}, an alternating least squares (ALS) technique is widely used~\cite{kolda2009tensor, MET}, which updates a factor matrix or a core tensor while keeping all others fixed.

		\begin{algorithm} [h!]
			\small
			\caption{Tucker-ALS} \label{alg:ALS_FULL}
			\SetKwInOut{Input}{Input}
			\SetKwInOut{Output}{Output}
			\Input{
				Tensor $\T{X} \in \mathbb{R}^{I_1 \times I_2 \times \cdots \times I_N}$, and \\
				core tensor dimensionality $J_1,...,J_N$. \\
			}	
			\Output{
				Updated factor matrices $\mathbf{A}^{(n)} \in \mathbb{R}^{I_n \times J_n}$  $(n=1, ... , N)$, and \\
				updated core tensor $\T{G} \in \mathbb{R}^{J_1 \times J_2 \times \cdots \times J_N}$. \\
			}
			\vspace{1.5mm}
			initialize all factor matrices $\mathbf{A}^{(n)}$\\
			\Repeat{the max. iteration or reconstruction error converges} {
				\For{$n=1...N$}{
					$\tensor{Y} \leftarrow \tensor{X}\times_{1}\mathbf{A}^{(1)\mat{T}} \cdots \times_{n-1}\mathbf{A}^{(n-1)\mat{T}}\times_{n+1}\mathbf{A}^{(n+1)\mat{T}}\cdots\times_{N}\mathbf{A}^{(N)\mat{T}} $ \\
					$\mathbf{A}^{(n)} \leftarrow J_{n}$ leading left singular vectors of $\T{Y}_{(n)}$
				}
			}
			$\tensor{G} \leftarrow \tensor{X}\times_{1}\mathbf{A}^{(1)\mathsf{T}} \cdots \times_{N}\mathbf{A}^{(N)\mathsf{T}}$
		\end{algorithm}

		Algorithm~\ref{alg:ALS_FULL} describes a conventional Tucker factorization based on the ALS, which is called the \textit{higher-order orthogonal iteration} (HOOI)  (see~\cite{kolda2009tensor} for details).
		The computational and memory bottleneck of  Algorithm~\ref{alg:ALS_FULL} is updating factor matrices $\mathbf{A}^{(n)}$ (lines 4-5), which requires tensor operations and SVD.
		Specifically, Algorithm~\ref{alg:ALS_FULL} requires storing a full-dense matrix $\T{Y}_{(n)}$, and the amount of memory needed for storing $\T{Y}_{(n)}$ is $O(I_n\prod_{m \neq n}{J_m})$. The required memory grows rapidly when the order, the dimensionality, or the rank of a tensor increase, and ultimately causes \textit{intermediate data explosion}~\cite{kang2012gigatensor}. Moreover, Algorithm~\ref{alg:ALS_FULL} computes SVD for a given $\T{Y}_{(n)}$, where the complexity of exact SVD is $O(\min(I_n\prod_{m \neq n}{J_{m}^2},I_{n}^2\prod_{m \neq n}{J_m}))$.
		The computational costs for SVD increase rapidly as well for a large-scale tensor.
		Notice that Algorithm~\ref{alg:ALS_FULL} assumes missing entries of $\T{X}$ as zeros during the update process (lines 4-5), and core tensor $\tensor{G}$ (line 7) is uniquely determined and relatively easy to be computed by an input tensor and factor matrices.

	 In summary, applying the naive Tucker-ALS algorithm on sparse tensors generates severe accuracy and scalability issues.
	 Therefore, Algorithm~\ref{alg:ALS_FULL} needs to be revised to focus only on observed entries and scale for large-scale tensors at the same time. In that case,
	 an alternative ALS approach is applicable to Algorithm~\ref{alg:ALS_FULL}, which is utilized for partially observable matrices~\cite{zhou2008large} and CP factorizations~\cite{ShinK17}. The alternative ALS approach is discussed in Section~\ref{sec:proposed_method}.
\end{definition}

\begin{definition}[Intermediate Data]
	We define intermediate data as memory requirements for updating $\mathbf{A}^{(n)}$ (lines 4-5 in Algorithm~\ref{alg:ALS_FULL}), excluding memory space for storing $\T{X}$, $\T{G}$, and $\mathbf{A}^{(n)}$.
	The size of intermediate data plays a critical role in determining which Tucker factorization algorithms are space-efficient, as we will discuss in Section~\ref{sec:complexity_analysis}.
\end{definition}

	\section{Proposed Method}
	\label{sec:proposed_method}

We describe \method, our proposed Tucker factorization algorithm for sparse tensors.
As described in Definition~\ref{def:ALS}, the computational and memory bottleneck of the standard Tucker-ALS algorithm occurs while updating factor matrices.
Therefore, it is imperative to update them efficiently in order to maximize scalability of the algorithm.
However, there are several challenges in designing an optimized algorithm for updating factor matrices.

\begin{enumerate}
	\item {\textbf{Exploit the characteristic of sparse tensors.}}
		Sparse tensors are composed of a vast number of missing entries and a  small number of observable entries.  How can we exploit the sparsity of given tensors to design an accurate and scalable algorithm for updating factor matrices?
	
	\item {\textbf{Maximize scalability.}}
		The aforementioned Tucker-ALS algorithm suffers from \textit{intermediate data explosion} and high computational costs while updating factor matrices. How can we formulate efficient algorithms for updating factor matrices in terms of time and memory?
	\item {\textbf{Parallelization.}}
		 It is crucial to avoid race conditions and adjust workloads between threads to thoroughly employ multi-core parallelism. How can we apply data parallelism on updating factor matrices in order to scale up linearly with respect to the number of threads?
\end{enumerate}

To overcome the above challenges, we suggest the following main ideas, which we describe in later subsections.

\begin{enumerate}
	\item {\textbf{Row-wise update rule}}
		fully exploits the sparsity of a given tensor and enhances the accuracy of a factorization (Figure~\ref{fig:mainplot} and Section~\ref{sec:method_factor_matrices}).
	\item {\textbf{\TOP and \APP}}
		accelerate the update process by caching intermediate calculations and truncating ``noisy'' entries from a core tensor,
		while \method itself provides a memory-optimized algorithm by default  (Section~\ref{sec:optimization}).
	\item {\textbf{Careful distribution of work}}
	assures that each thread has independent tasks and balanced workloads when \method updates factor matrices. (Section~\ref{sec:parallel}).
\end{enumerate}

	We first suggest an overview of how \method factorizes sparse tensors using Tucker method in Section~\ref{sec:overview}.
After that, we describe details of our main ideas in Sections~\ref{sec:method_factor_matrices}$\thicksim$\ref{sec:parallel}, and we offer a theoretical analysis of \method in Section~\ref{sec:analysis}.

\vspace{-3mm}

\subsection{Overview}
\label{sec:overview}

\method provides an efficient Tucker factorization algorithm for sparse tensors.

\begin{figure}[h!]
	\centering
	\includegraphics[height=6.5cm, width=9cm]{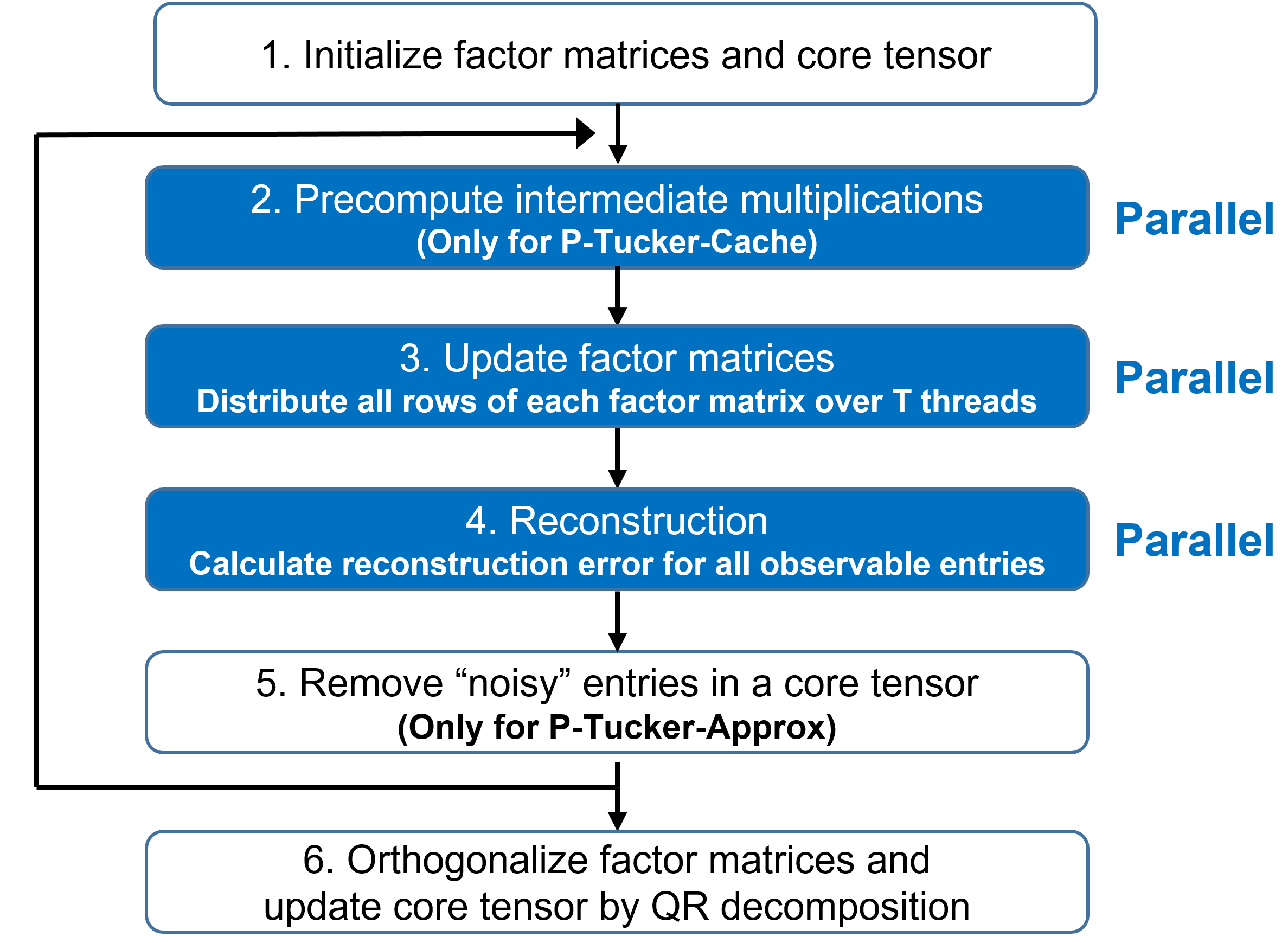}
	\caption{
			An overview of \method. After initialization, \method updates factor matrices in a fully-parallel way. When the reconstruction error converges, \method performs QR decomposition to make factor matrices orthogonal and updates a core tensor.
	}
	\label{fig:workflow}
\end{figure}

 Figure~\ref{fig:workflow}  and Algorithm~\ref{alg:partial} describe the main process of \method. First, \method initializes all $\mathbf{A}^{(n)}$ and $\T{G}$ with random real values between 0 and 1 (step 1 and line 1).
After that, \method updates factor matrices (steps 2-3 and line 3) by Algorithm~\ref{alg:proposed} explained in Section~\ref{sec:method_factor_matrices}.
	When all factor matrices are updated, \method measures reconstruction error using~\eqref{eq:reconfull} (step 4 and line 4). In case of \APP (step 5 and lines 5-6), \APP removes ``noisy" entries of $\T{G}$ by Algorithm~\ref{alg:APP} explained in Section~\ref{sec:optimization}. \method stops iterations if the error converges or the maximum iteration is reached (line 7).
Finally, \method performs QR decomposition on all $\mathbf{A}^{(n)}$ to make them orthogonal and updates $\T{G}$ (step 6 and lines 8-11).
Specifically, QR decomposition~\cite{linear_algebra} on each $\mathbf{A}^{(n)}$ is defined as follows:
\begin{equation}
\mathbf{A}^{(n)}=\mathbf{Q}^{(n)}\mathbf{R}^{(n)},~n=1...N \\
\end{equation}
where $\mathbf{Q}^{(n)}\in\mathbb{R}^{I_n \times J_n}$ is \textit{column-wise orthonormal} and $\mathbf{R}^{(n)}\in\mathbb{R}^{J_n \times J_n}$ is \textit{upper-triangular}.
Therefore, by substituting $\mathbf{Q}^{(n)}$ for $\mathbf{A}^{(n)}$, \method succeeds in making factor matrices orthogonal.
Core tensor $\T{G}$ must be updated accordingly in order to maintain the same reconstruction error. According to~\cite{multilinear_kolda}, the update rule of core tensor $\T{G}$ is given as follows:
\begin{equation}
\T{G} \leftarrow \T{G} \times_{1} \mathbf{R}^{(1)} \cdots \times_{N} \mathbf{R}^{(N)}.
\end{equation}

\begin{algorithm} [t!]
	\small
	\caption{\method for Sparse Tensors} \label{alg:partial}
	\SetKwInOut{Input}{Input}
	\SetKwInOut{Output}{Output}
	\Input{
		Tensor $\T{X} \in \mathbb{R}^{I_1 \times I_2 \times \cdots \times I_N}$, \\
		core tensor dimensionality $J_1,...,J_N$, and \\
		truncation rate $p$ (\APP only).\\
	}	
	\Output{
		Updated factor matrices $\mathbf{A}^{(n)} \in \mathbb{R}^{I_n \times J_n} (n=1,...,N)$, \\
		and updated core tensor $\T{G} \in \mathbb{R}^{J_1 \times J_2 \times \cdots \times J_N}$. \\
	}
	initialize factor matrices $\mathbf{A^{(n)}}~(n=1,...,N)$ and core tensor $\tensor{G}$ \\
	\Repeat{the maximum iteration or $\|\tensor{X}-\tensor{X'}\|$ converges}{
		update factor matrices $\mathbf{A}^{(n)}~(n=1,...,N)$ by Algorithm~\ref{alg:proposed} \\
		calculate reconstruction error using~\eqref{eq:reconfull} \\
		\If(\Comment*[f]{$\T{G}$ Truncation}){\APP}{\label{eq:main_TR}
			remove ``noisy" entries of $\T{G}$ by Algorithm~\ref{alg:APP}
		}
	}
	\For{n = 1...$N$}{
		$\mathbf{A}^{(n)} \rightarrow \mathbf{Q}^{(n)}\mathbf{R}^{(n)}$\Comment*[f]{QR decomposition}\\
		$\mathbf{A}^{(n)} \leftarrow \mathbf{Q}^{(n)}$\Comment*[f]{Orthogonalize $\mathbf{A}^{(n)}$}\\
		$\tensor{G} \leftarrow \tensor{G}\times_{n}\mathbf{R}^{(n)}$\Comment*[f]{Update core tensor $\T{G}$}
	}
\end{algorithm}

\begin{figure*}[t!]
	\centering
	\vspace{-3mm}
	\includegraphics[height=7cm]{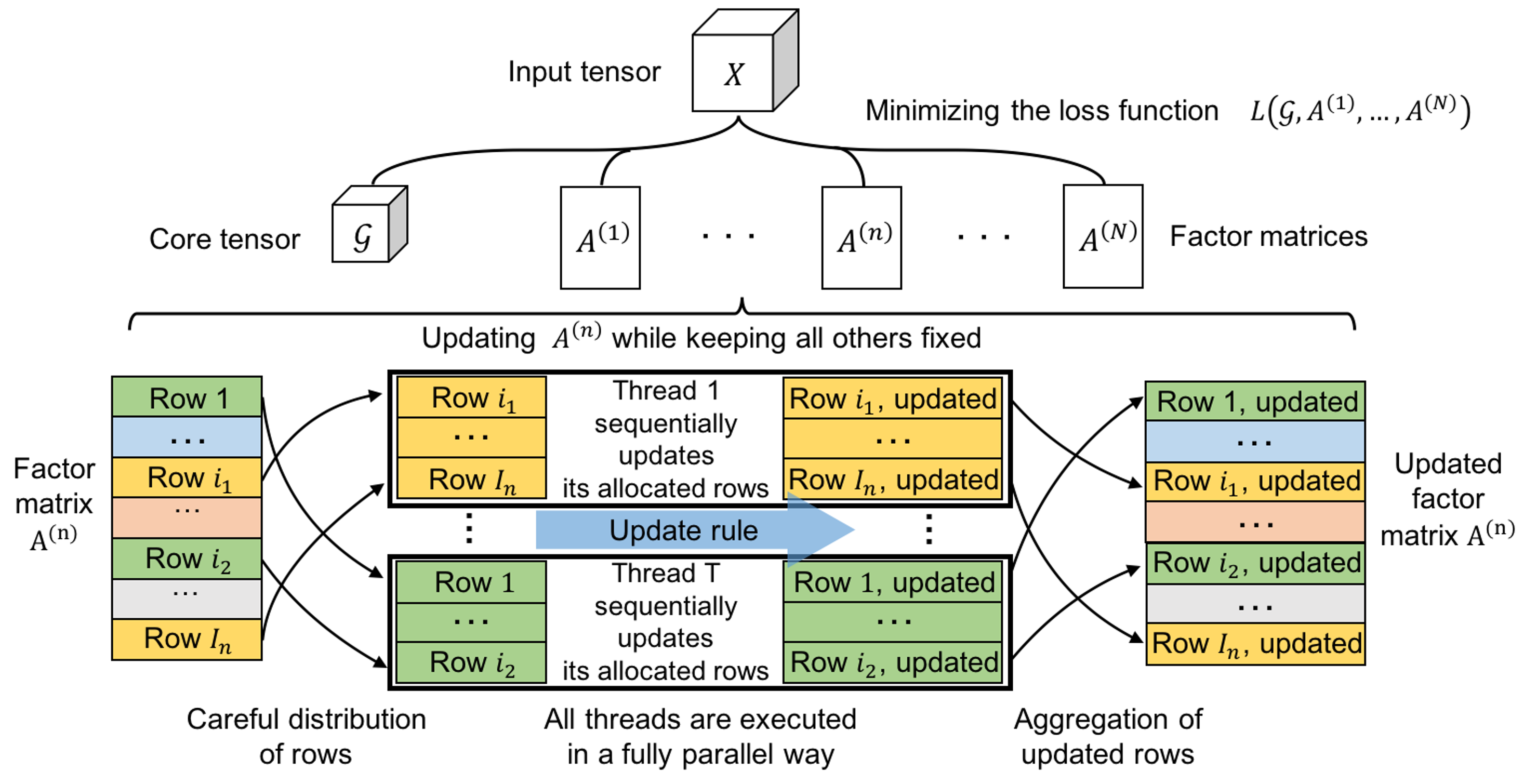}
	\caption{
		An overview of updating factor matrices. \method performs a row-wise ALS which updates each row of a factor matrix $\mathbf{A}^{(n)}$ while keeping all the others fixed. Since all rows of a factor matrix are independent of each other in terms of minimizing the loss function~(\ref{eq:TF_PARTIAL}), \method fully exploits multi-core parallelism to update all rows of $\mathbf{A}^{(n)}$.
		First, all rows are carefully distributed to all threads to achieve a uniform workload among them. After that, all threads update their allocated rows in a fully parallel way. In a single thread, the allocated rows are updated in a sequential way. Finally, \method aggregates all updated rows from all threads to update $\mathbf{A}^{(n)}$. \method iterates this update procedure for all factor matrices one by one.
	}
	\label{fig:mainplot}
\end{figure*}


\vspace{-3mm}

\subsection{Row-wise Updates of Factor Matrices}
\label{sec:method_factor_matrices}

\method updates factor matrices in a row-wise manner based on ALS, where an update rule for a row is computed by only observed entries of a tensor.
From a high-level point of view, as most ALS methods do, \method updates a factor matrix at a time while maintaining all others fixed. However,
when all other matrices are fixed, there are several approaches~\cite{ShinK17} for updating a single factor matrix.
Among them, \method selects a row-wise update method;
a key benefit of the row-wise update is that all rows of a factor matrix are independent of each other in terms of minimizing the loss function~(\ref{eq:TF_PARTIAL}). This property enables applying multi-core parallelism on updating factor matrices.
Given a row of a factor matrix, an update rule is derived by computing a gradient with respect to the given row and setting it as zero, which minimizes the loss function~\eqref{eq:TF_PARTIAL}.
The update rule for the $i_n$th row of the $n$th factor matrix $\mat{A}^{(n)}$ (see Figure~\ref{fig:Factor_one}) is given as follows; the proof of Equation~\eqref{eq:rowupdate} is in Theorem~\ref{theorem1}.
			\begin{multline} \label{eq:rowupdate}
			\begin{aligned}
			[a^{(n)}_{i_{n}1}, ..., a^{(n)}_{i_{n}J_{n}}] \leftarrow \argmin{[a^{(n)}_{i_{n}1}, ..., a^{(n)}_{i_{n}J_{n}}]}{L(\T{G},\mat{A}^{(1)},...,\mat{A}^{(N)})} \\= \vect{c}_{i_n:}^{(n)} \times [\mat{B}_{i_n}^{(n)}+\lambda \mathbf{I}_{J_n}]^{-1}
			\end{aligned}
			\end{multline}
			where $\mat{B}_{i_n}^{(n)}$ is a  ${J_n \times J_n}$ matrix whose $(j_1,j_2)$th entry is
			\begin{equation} \label{eq:rowB}
			 \sum_{\forall(i_1,...,i_N)\in\Omega_{i_n}^{(n)}}\delta_{(i_1,...,i_N)}^{(n)}(j_1) \delta_{(i_1,...,i_N)}^{(n)}(j_2),
			\end{equation}
			$\vect{c}_{i_n:}^{(n)}$ is a length ${J_n}$ vector whose $j$th entry is
			\begin{equation} \label{eq:rowC}
			\sum_{\forall(i_1,...,i_N)\in\Omega_{i_n}^{(n)}}\tensor{X}_{(i_1,...,i_N)}  \delta_{(i_1,...,i_N)}^{(n)}(j),
			\end{equation}
			$\delta_{(i_1,...,i_N)}^{(n)}$ is a length ${J_n}$ vector whose $j$th entry is
			\begin{equation} \label{eq:delta}
			 \sum_{\forall(j_1...j_n=j...j_N)\in\T{G}}\tensor{G}_{(j_1...j_n=j...j_N)}\prod_{k \neq n} a^{(k)}_{i_k j_k},
			\end{equation}

			\begin{figure}[t!]
				\centering
				\includegraphics[width=8.5cm]{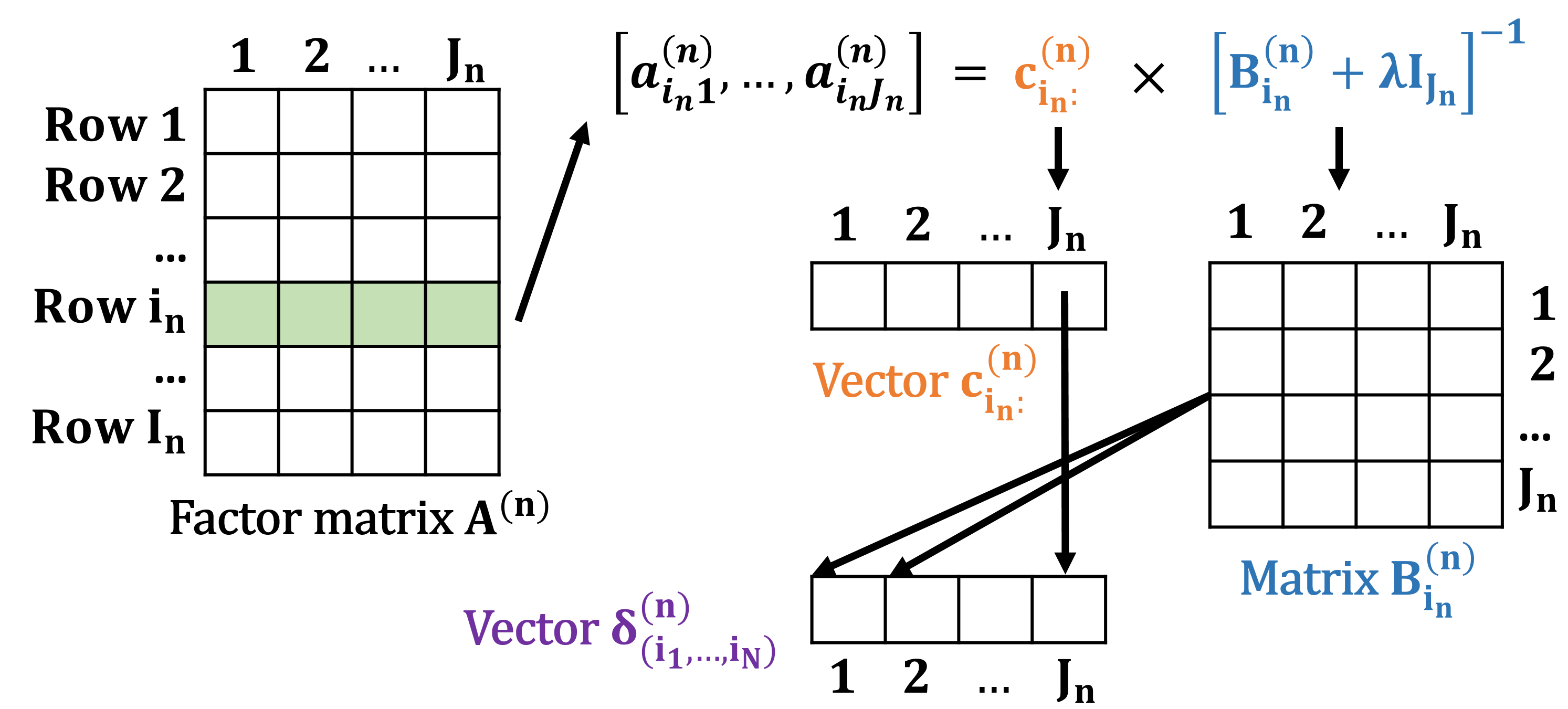}
				\vspace{2mm}
				\caption{
					An illustration of an update rule for a row of a factor matrix. \method requires three intermediate data $\mat{B}_{i_n}^{(n)}$, $\vect{c}_{i_n:}^{(n)}$, and $\delta_{(i_1,...,i_N)}^{(n)}$ for updating the $i_n$th row of $\mathbf{A^{(n)}}$.  Note that $\lambda$ is a regularization parameter, and $\mathbf{I}_{J_n}$ is a $J_n \times J_n$ identity matrix.
				}
				\label{fig:Factor_one}
			\end{figure}

\noindent $\Omega_{i_n}^{(n)}$ indicates the subset of $\Omega$ whose $n$th mode's index is $i_n$, $\lambda$ is a regularization parameter, and $\mathbf{I}_{J_n}$ is a $J_n \times J_n$ identity matrix.
As shown in Figure~\ref{fig:Factor_one}, the update rule for the $i_n$th row of $\mat{A}^{(n)}$ requires three intermediate data $\mat{B}_{i_n}^{(n)}$, $\vect{c}_{i_n:}^{(n)}$, and $\delta_{(i_1,...,i_N)}^{(n)}$. Those data are computed by the subset of observable entries $\Omega_{i_n}^{(n)}$. Thus, computational costs of updating factor matrices are proportional to the number of observable entries,
which lets \method  fully exploit the sparsity of given tensors.
Moreover, \method predicts missing values of a tensor using~\eqref{eq:recon2}, not as zeros.  Equation~\eqref{eq:recon2} is computed by updated factor matrices and a core tensor, and they are learned by observed entries of a tensor. Hence, \method not only enhances the accuracy of factorizations, but also reflects the latent-characteristics of observed entries of a tensor.
Note that a matrix $[\mat{B}_{i_n}^{(n)}+\lambda \mathbf{I}_{J_n}]$ is positive-definite and invertible, and a proof of the update rule is summarized in Section~\ref{sec:convergence}.

			\begin{algorithm} [t!]
				\footnotesize
				\caption{\method for Updating Factor Matrices} \label{alg:proposed}
				\SetKwInOut{Input}{Input}
				\SetKwInOut{Output}{Output}
				\Input{
					Tensor $\T{X} \in \mathbb{R}^{I_1 \times I_2 \times \cdots \times I_N}$, \\
					factor matrices $\mathbf{A}^{(n)} \in \mathbb{R}^{I_n \times J_n}~(n=1, ... , N)$,
					\\core tensor $\T{G} \in \mathbb{R}^{J_1 \times J_2 \times \cdots \times J_N}$ , and\\
					cache table $Pres \in \mathbb{R}^{|\Omega| \times |\T{G}|}$ (\TOP only).\\
				}
				\Output{
					Updated factor matrices $\mathbf{A}^{(n)} \in \mathbb{R}^{I_n \times J_n}~(n=1, ... ,N)$.
				}
					\If(\Comment*[f]{Precompute $Pres$}){\TOP}{
						\For(\Comment*[f]{In parallel}){${\alpha=\forall(i_1,...,i_N)}\in\Omega$}{ \label{eq:pres1}
							\For{${\beta=\forall(j_1,...,j_N)}\in\tensor{G}$}{
								$ Pres[\alpha][\beta] \leftarrow \T{G}_{\beta}\prod_{k=1}^{N}{a_{i_k j_k}^{(k)}}$  \label{eq:pres2}
							}
						}
					}
					\For(\Comment*[f]{nth factor matrix}){n = 1...$N$}{ \label{eq:nth}
						\For(\Comment*[f]{$i_n$th row, in parallel}){$i_{n}$ = 1...$I_{n}$}{ \label{eq:row}
							\For{$\alpha=\forall(i_1,...,i_N)\in\Omega_{i_n}^{(n)}$}{ \label{eq:nnz}
								\For(\Comment*[f]{Compute $\delta$}){$\beta=\forall(j_1,...,j_N)\in\tensor{G}$}{
									\If{\MOP}{
										$\delta_{\alpha}^{(n)}(j_n) \leftarrow \delta_{\alpha}^{(n)}(j_n)  +
										\tensor{G}_{\beta} \prod_{k \neq n}a_{i_k j_k}^{(k)}$\\ \label{eq:MOP_delta}
									}
									\If{\TOP}{
										$\delta_{\alpha}^{(n)}(j_n) \leftarrow \delta_{\alpha}^{(n)}(j_n)  + \frac{ Pres[\alpha][\beta]}{a_{i_n j_n}^{(n)}}$\\ \label{eq:TOP_delta}
									}
								}
								calculate $\mat{B}_{i_n}^{(n)}$ and $\vect{c}_{i_n :}^{(n)}$ using \eqref{eq:rowB} and \eqref{eq:rowC}\label{eq:main_BC}\\
							}
							find the inverse matrix of $[\mat{B}_{i_n}^{(n)}+\lambda \mathbf{I}_{J_n}]$\label{eq:main_inverse}\\
							update  $[a_{i_n 1}^{(n)}, \cdots, a_{i_n J_n}^{(n)}]$ using \eqref{eq:rowupdate}\label{eq:main_rowupdate}	
						}
						\If(\Comment*[f]{Update $Pres$}){\TOP}{ \label{eq:main_pres1}
							\For(\Comment*[f]{In parallel}){$\alpha=\forall(i_1,...,i_N)\in\Omega$}{
								\For{$\beta=\forall(j_1,...,j_N)\in\tensor{G}$}{
									$Pres[\alpha][\beta]\leftarrow \frac{Pres[\alpha][\beta]}{(a,old)_{i_n j_n}^{(n)}} \times (a,new)_{i_n j_n}^{(n)}$ \label{eq:update_pres}
								}
							}
						}
					}
			\end{algorithm}

Algorithm~\ref{alg:proposed} describes how \method updates factor matrices.
First, in case of \TOP (lines 1-4), it computes the values of all entries in a cache table $Pres$ $(\in \mathbb{R}^{|\Omega| \times |\T{G}|})$ which caches intermediate multiplication results generated while updating factor matrices. This memoization technique allows \TOP to be a time-efficient algorithm.
Next, \method chooses a row $\vect{a}_{i_n :}^{(n)}$ of a factor matrix $\mat{A}^{(n)}$ to update (lines 5-6).
After that, \method computes $\mat{B}_{i_n}^{(n)}$ and $\vect{c}_{i_n:}^{(n)}$ required for updating a row $\vect{a}_{i_n :}^{(n)}$ (lines 7-13).
\method performs matrix inverse operation on  $[\mat{B}_{i_n}^{(n)}+\lambda \mathbf{I}_{J_n}]$ (line 14) and updates a row $\vect{a}_{i_n :}^{(n)}$ by the multiplication of $\vect{c}_{i_n:}^{(n)}$ and  $[\mat{B}_{i_n}^{(n)}+\lambda \mathbf{I}_{J_n}]^{-1}$ (line 15).
In case of \TOP, it recalculates $Pres$ using the existing and updated $\mathbf{A}^{(n)}$ (lines 16-19) whenever $\mathbf{A}^{(n)}$ is updated.
Note that $\alpha$ and $\beta$ indicate an entry of $\tensor{X}$ and $\tensor{G}$, respectively.
\subsection{Variants: \TOP and \APP} \label{sec:optimization}
As discussed in Section~\ref{sec:method_factor_matrices}, \method requires three intermediate data: $\mat{B}_{i_n}^{(n)}$, $\vect{c}_{i_n:}^{(n)}$, and $\delta_{(i_1,...,i_N)}^{(n)}$ whose memory requirements are $O(J_n^2)$. Considering the memory complexity of the naive Tucker-ALS, which is $O(I_n\prod_{m \neq n}{J_m})$, \method successfully provides a memory-optimized algorithm. We can further optimize \method in terms of time by a caching algorithm (\TOP) and an approximation algorithm (\APP).

The crucial difference between \MOP and \TOP lies in the computation of the intermediate vector $\delta$ (lines 9-12 in
Algorithm~\ref{alg:proposed}).
In case of \MOP, updating $\delta$ requires $N$ times of multiplications for a given entry pair $(\alpha,\beta)$ (line 10), which takes $O(N)$.
 However, if we  cache the results of those multiplications for all entry pairs, the update only takes $O(1)$ (line 12).
 This trade-off distinguishes \TOP and \MOP.
 \TOP accelerates intermediate calculations by the memoization technique with the cache table $Pres$.
 Meanwhile, \MOP requires only small vectors $\vect{c}_{i_n:}^{(n)}$ and $\delta_{(i_1,...,i_N)}^{(n)}$ ($\in\mathbb{R}^{J_n}$) and
 a small matrix $\mat{B}_{i_n}^{(n)}$ ($\in\mathbb{R}^{J_n \times J_n}$) as intermediate data.
 Note that when $a_{i_n j_n}^{(n)}$ is 0 (lines 12 and 19), \TOP conducts the multiplications as \MOP does (line 10).
 
\begin{figure}[b!]
	\centering
	\vspace{2mm}
	\begin{subfigure}[t]{0.23\textwidth}
		\includegraphics[width=4.3cm,height=3.6cm]{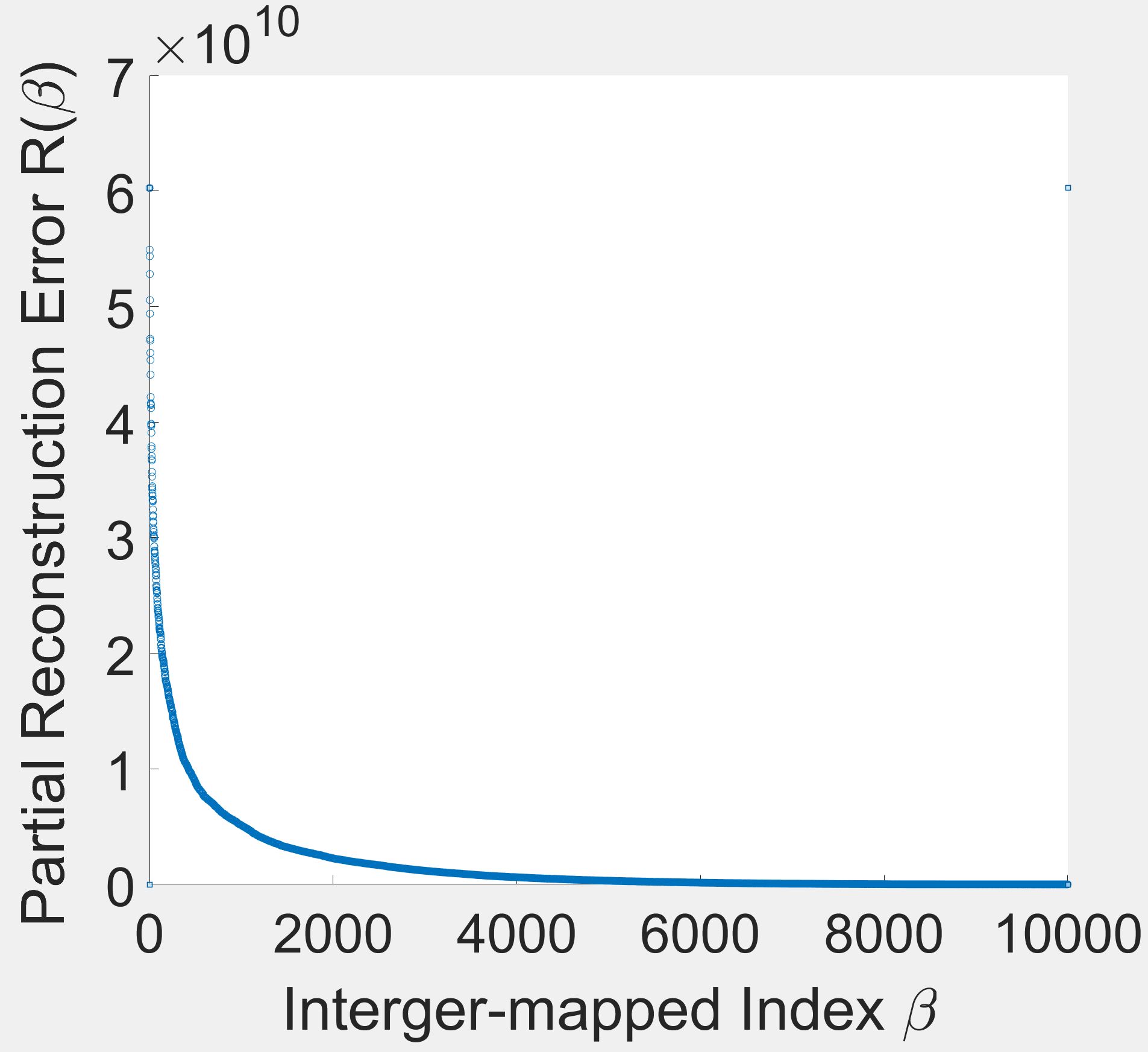}
		\label{fig:RBvalue}
	\end{subfigure}
	\hfill
	\begin{subfigure}[t]{0.23\textwidth}
		\includegraphics[width=4.3cm]{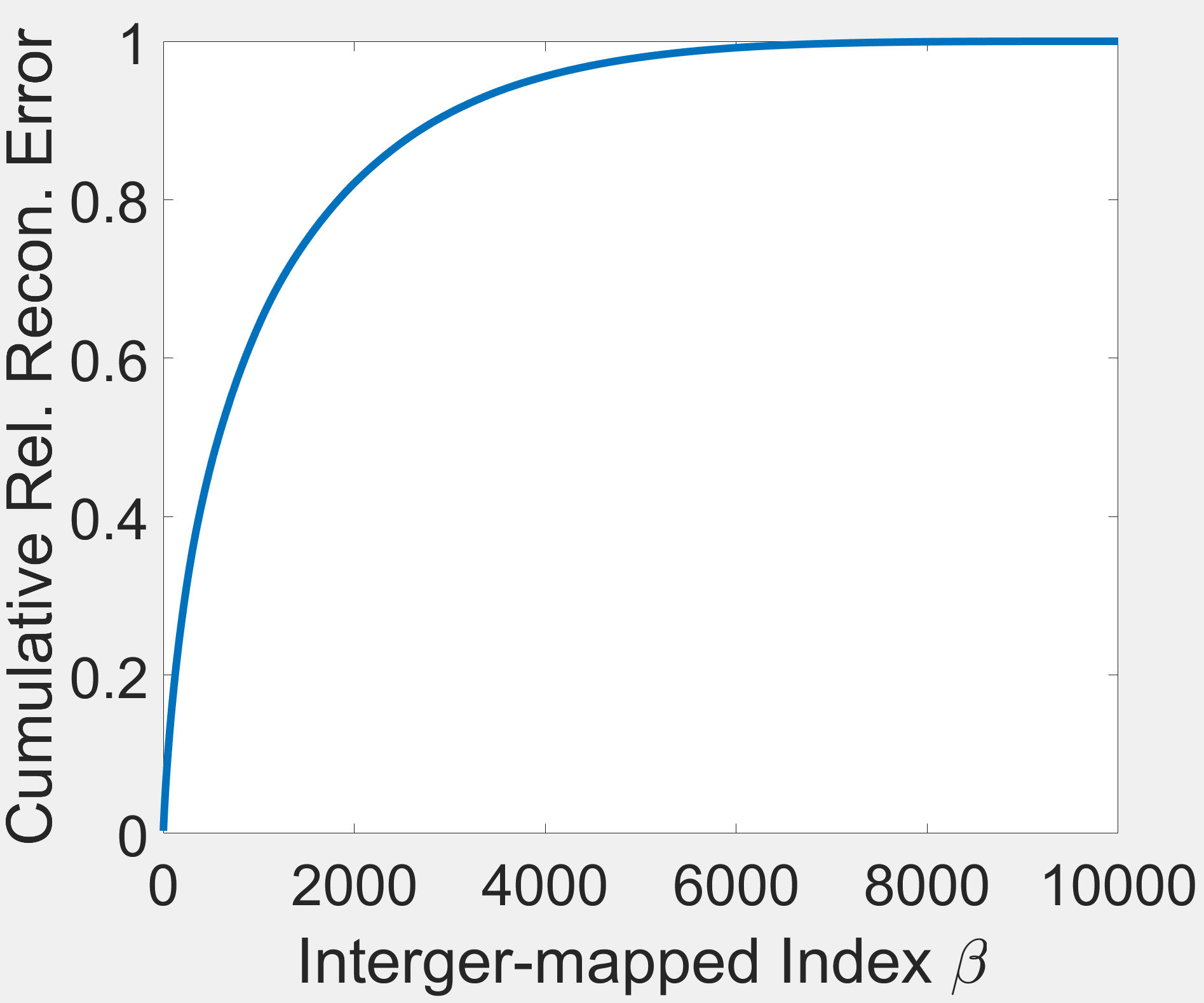}
		\label{fig:CRB}
	\end{subfigure}
	\vspace{-2mm}
	\caption{Distribution of partial reconstruction error $\T{R}(\beta)$ and accumulation of relative reconstruction error produced by an entry $\beta=(j_1,...,j_N)$ of a core tensor $\T{G}$. Note that 20\% ``noisy`` entries of $\T{G}$ generate 80\% of total reconstruction error.}
	\label{fig:APPgraph}
\end{figure}

 The main intuition of \APP is that there exist ``noisy'' entries in a core tensor $\T{G}$, and we can accelerate the update process by truncating these ``noisy'' entries of $\T{G}$.
 Then, how can we determine whether an entry of $\T{G}$ is ``noisy" or not?
 A naive approach could be treating an entry $(j_1,...,j_N)\in\T{G}$ with small $\T{G}_{(j_1,...,j_N)}$ value as "noisy" like the truncated SVD~\cite{Hansen1987}. However, in this case, small-value entries are not always negligible since their contributions to minimizing the error \eqref{eq:reconfull} can be larger than that of large-value ones.
Hence, we propose more precise criterion which regards an entry $\beta=(j_1,...,j_N)\in\T{G}$ with a high $\T{R}(\beta)$ value as ``noisy''.
$\T{R}(\beta)$ indicates a partial reconstruction error  produced by an entry $\beta$, derived from the sum of terms only related to $\beta$ in \eqref{eq:reconfull}.
Given an entry $\beta=(j_1,...,j_N)\in\T{G}$, $\T{R}(\beta)$ is given as follows:
\begin{multline} \label{eq:recon_partial}
\footnotesize
\hspace{-5mm}
\sum_{\forall{\alpha}\in\Omega}\Bigg({\bigg(\tensor{X}_{\alpha}-\sum_{\forall{\gamma}\in\T{G}}\tensor{G}_{\gamma}\prod_{n=1}^{N}a^{(n)}_{i_{n}j_n}\bigg)^{2}} -~ {\bigg(\tensor{X}_{\alpha}-\sum_{\forall{\gamma} \neq \beta}\tensor{G}_{\gamma}\prod_{n=1}^{N}a^{(n)}_{i_{n}j_n}\bigg)^{2}}\Bigg) =\\
\footnotesize
\hspace{-2mm}
 \sum_{\forall{\alpha}\in\Omega}{(\tensor{G}_{\beta}\prod_{n=1}^{N}a^{(n)}_{i_{n}j_n})\left(-2\tensor{X}_{\alpha}+\tensor{G}_{\beta}\prod_{n=1}^{N}a^{(n)}_{i_{n}j_n}+2\sum_{\forall{\gamma}\neq\beta}\tensor{G}_{\gamma}\prod_{n=1}^{N}a^{(n)}_{i_{n}j_n}\right)}.
\hspace{-5mm}
\end{multline}

Note that we use $\alpha$, $\beta$, and $\gamma$ symbols to simplify the equation.
$\T{R}(\beta)$ suggests a more precise guideline of ``noisy'' entries since $\T{R}(\beta)$ is a part of \eqref{eq:reconfull}, while the naive approach assumes the error based on the value $\T{G}_{(j_1,...,j_N)}$.
Figure~\ref{fig:APPgraph} illustrates a distribution of $\T{R}(\beta)$ and a cumulative function of relative reconstruction error on the latest MovieLens dataset ($J=10$). As expected by our intuition, only 20\% entries of $\T{G}$ generate about 80\% of total reconstruction error. Algorithm~\ref{alg:APP} describes how \APP truncates  ``noisy'' entries in $\T{G}$. It first computes $\T{R}(\beta)$ (lines 1-2) for all entries in $\T{G}$, and sort $\T{R}(\beta)$ in descending order (line 3) as well as their indices. Finally, it truncates top-$p|\T{G}|$  ``noisy'' entries of $\T{G}$ (line 4).
\APP performs Algorithm~\ref{alg:APP} for each iteration (lines 3-6 in Algorithm~\ref{alg:partial}), which reduces the number of non-zeros in $\T{G}$ step-by-step. Therefore, the elapsed time per iteration also decreases since the time complexity of \APP depends on the number of non-zeros $|\T{G}|$.
Practically, we note that \APP may require few iterations to run faster than \method due to overheads from calculating $\T{R}(\beta)$, which is computed for all iterations.

\begin{algorithm} [t!]
	\small
	\caption{Removing noisy entries of a core tensor $\T{G}$ in \APP} \label{alg:APP}
	\SetKwInOut{Input}{Input}
	\SetKwInOut{Output}{Output}
	\Input{
		Tensor $\T{X} \in \mathbb{R}^{I_1 \times I_2 \times \cdots \times I_N}$, \\
		factor matrices $\mathbf{A}^{(n)} \in \mathbb{R}^{I_n \times J_n} (n=1, ... ,N)$, \\
		core tensor $\T{G} \in \mathbb{R}^{J_1 \times J_2 \times \cdots \times J_N}$, and\\
		truncation rate $p~(0 < p < 1).$\\
	}	
	\Output{
		Truncated core tensor $\T{G'} \in \mathbb{R}^{J_1 \times J_2 \times \cdots \times J_N}$. \\
	}
	\For{$\beta=\forall(j_1,...,j_N)\in\tensor{G}$}{
		compute a partial reconstruction error $\T{R}(\beta)$ by~\eqref{eq:recon_partial}
	}
	sort $\T{R}(\beta)$ in descending order with their indices \\
	remove $p|\T{G}|$ entries in $\T{G}$, whose $\T{R}(\beta)$ value are ranked within top-$p|\T{G}|$ among all $\T{R}(\beta)$ values.
\end{algorithm}


With the above optimizations, \method becomes the most time and memory efficient method in theoretical and experimental perspectives (see Table~\ref{table::analysis}).

%

\subsection{Careful Distribution of Work}
\label{sec:parallel}
	There are three sections where multi-core parallelization is applicable in Algorithms~\ref{alg:partial} and \ref{alg:proposed}.
	The first section (lines 2-4 and 17-19 in Algorithm~\ref{alg:proposed}) is for \TOP when it computes and updates the cache table $Pres$.
	 The second section (lines 6-15 in Algorithm~\ref{alg:proposed}) is for updating factor matrices, and the last section (line 4 in Algorithm~\ref{alg:partial}) is for measuring the reconstruction error.
	 For each section, \method carefully distributes tasks to threads while maintaining the independence between them. Furthermore, \method utilizes a dynamic scheduling method~\cite{openmp} to assure that each thread has balanced workloads, which directly affects the performance (see Section~\ref{sec:exp:thread_scalability}).
	The details of how \method parallelizes each section are as follows. Note that $T$ indicates the number of threads used for parallelization.
	\begin{itemize}
		\item{\textbf{Section 1: Computing and Updating Cache Table $Pres$ (Only for \TOP).}
	All rows of $Pres$ are independent of each other when they are computed or updated.
	Thus, \method distributes all rows equally over $T$ threads, and each thread computes or updates allocated rows independently using static scheduling.
	}
	\item{\textbf{Section 2: Updating Factor Matrices.}
	All rows of $\mathbf{A}^{(n)}$ are independent of each other regarding minimizing the loss function~(\ref{eq:TF_PARTIAL}). Therefore, \method distributes all rows uniformly to each thread, and updates them in parallel. Since $|\Omega_{i_n}^{(n)}|$ differs for each row, the workload of each thread may vary considerably. Thus, \method employs dynamic scheduling in this part.
	}
	\item{ \textbf{Section 3: Calculating Reconstruction Error.}
	All observable entries are independent of each other in measuring the reconstruction error. Thus, \method distributes them evenly over $T$ threads, and each thread computes the error separately using static scheduling. At the end, \method aggregates the partial error from each thread.
	}	
\end{itemize}

\subsection{Theoretical Analysis}
\label{sec:analysis}

\subsubsection{Convergence Analysis}
\label{sec:convergence}

We theoretically prove the correctness and the convergence of \method.
\begin{theorem}[Correctness of \method] \label{theorem1}
	The proposed row-wise update rule~\eqref{eq:theorem1} minimizes the loss function~\eqref{eq:TF_PARTIAL} regarding the updated parameters.
	\begin{equation} \label{eq:theorem1}
	\small
	\argmin{[a^{(n)}_{i_{n}1}, ..., a^{(n)}_{i_{n}J_{n}}]}{L(\T{G},\mat{A}^{(1)},...,\mat{A}^{(N)})} = \vect{c}_{i_n:}^{(n)} \times [\mat{B}_{i_n}^{(n)}+\lambda \mathbf{I}_{J_n}]^{-1}
	\end{equation}
\end{theorem}
\begin{proof}
	\footnotesize
	\[
	\frac{\partial L}{\partial a_{i_n j_n}^{(n)}} = 0, \forall{j_n}, 1 \leq j_n \leq J_{n} \quad
	\]
	\[
	\Leftrightarrow \sum_{\forall{\alpha}\in\Omega_{i_n}^{(n)}} \Bigg(\bigg( \T{X}_{\alpha}-\sum_{\forall{\beta}\in\T{G}}\tensor{G}_{\beta}\prod_{n=1}^{N}a^{(n)}_{i_{n}j_n} \bigg) \times \bigg(-\delta_{\alpha}^{(n)}(j_n)\bigg)\Bigg) + \lambda a_{i_nj_n}^{(n)} = 0
	\]
	\[
	\hspace{-4mm}
	\Leftrightarrow [a^{(n)}_{i_{n}1}, ..., a^{(n)}_{i_{n}J_{n}}] \Bigg( \sum_{\forall{\alpha}\in\Omega_{i_n}^{(n)}} \bigg( \delta_{\alpha}^{(n)\mat{T}} \delta_{\alpha}^{(n)} \bigg)  + \lambda \mathbf{I}_{J_n} \Bigg)
	=\sum_{\forall{\alpha}\in\Omega_{i_n}^{(n)}} \Bigg( \T{X}_{\alpha} \delta_{\alpha}^{(n)}\Bigg)
	\]
	\[
	\Leftrightarrow [a^{(n)}_{i_{n}1}, ..., a^{(n)}_{i_{n}J_{n}}] = \vect{c}_{i_n:}^{(n) } \times [\mat{B}_{i_n}^{(n)}+\lambda \mathbf{I}_{J_n}]^{-1}
	\]
\end{proof}

Note that the full proof of Theorem~\ref{theorem1} is in the supplementary material of \method~\cite{supple}.
\vspace{2mm}
\begin{theorem}[Convergence of \method]
	\method converges since \eqref{eq:TF_PARTIAL} is bounded and decreases monotonically.
\end{theorem}
\begin{proof}
	According to Theorem~\ref{theorem1}, the loss function~\eqref{eq:TF_PARTIAL} never increases since every update in \method minimizes it, and \eqref{eq:TF_PARTIAL} is bounded by 0. Thus, \method converges.
\end{proof}

\begin{table}[t!]
	\centering
	\caption{Complexity analysis of \method and other methods with respect to time and memory. The optimal complexities are in bold. \method and its variants exhibit the best time and memory complexity among all methods. Note that memory complexity indicates the space requirement for intermediate data.
	}	
	\begin{tabular}{c|cc}
		\toprule
		\textbf{Algorithm} & \textbf{Time Complexity}  & \textbf{Memory} \\
		& (per iteration) & \textbf{Complexity}  \\
		\midrule
		\MOP   & $O(NIJ^3+N^2|\Omega|J^{N})$ & $\mathbf{O(TJ^2)}$ \\
		\TOP   & $O(NIJ^3+N|\Omega|J^{N})$ & $O(|\Omega|J^{N})$ \\
		\APP   & $\mathbf{O(NIJ^3+N^2|\Omega||\T{G}|)}$ & $O(J^N)$ \\
		\wopt~\cite{filipovic2015tucker}  & $O(N\sum_{k=0}^{k=N}(I^{N-k}J^k))$  & $O(I^{N-1}J)$ \\
		\CSF~\cite{smith2017tucker} & $O(NJ^{N-1}(|\Omega|+J^{2(N-1)}))$ & $O(IJ^{N-1})$ \\
		\SHOT~\cite{Oh:2017:SHOT}  & $O(NJ^{N}+N|\Omega|J^{N})$  & $O(J^{N-1})$ \\
		\bottomrule
	\end{tabular}
	\label{table::analysis}
\end{table}

\subsubsection{Complexity Analysis}
\label{sec:complexity_analysis}

We analyze time and memory complexities of \method and its variants. For simplicity, we assume $I_1=...=I_N=I$ and $J_1=...=J_N=J$. Table~\ref{table::analysis} summarizes the time and memory complexities of \method and other methods.
As expected in Section~\ref{sec:optimization}, \method presents the best memory complexity among all algorithms. While \TOP shows better time complexity than that of \method, \APP exhibits the best time complexity thanks to the reduced number of non-zeros in $\T{G}$.
Note that we calculate time complexities per iteration (lines 3-6 in Algorithm~\ref{alg:partial}),
and we focus on memory complexities of intermediate data, not of all variables.

\begin{theorem}[Time complexity of \MOP] \label{lemma:time_MOP} 	
	The time complexity of \MOP is $O(NIJ^3+N^2|\Omega|J^N)$.
\end{theorem}
\begin{proof} \label{proof:time_MOP}
		 Given the $i_n$th row of $\mathbf{A}^{(n)}$ (lines 5-6) in Algorithm~\ref{alg:proposed} ,
		 computing $\delta_{\alpha}^{(n)}(j_n)$ (line 10) takes $O(N|\Omega_{i_n}^{(n)}|J^N)$.
		 Updating $\mat{B}_{i_n}^{(n)}$ and $\vect{c}_{i_n :}^{(n)}$ (line 13) takes $O(|\Omega_{i_n}^{(n)}|J^2)$ since  $\delta_{\alpha}^{(n)}$ is already calculated.
		  Inverting $[\mat{B}_{i_n}^{(n)}+\lambda \mat{I}_{J_n}]$ (line 14) takes $O(J^3)$,
		  and updating a row (line 15) takes $O(J^2)$.
		  Thus, the time complexity of updating the $i_n$th row of $\mathbf{A}^{(n)}$ (lines 7-15) is $O(J^3+N|\Omega_{i_n}^{(n)}|J^N)$. Iterating it for all rows of $\mathbf{A}^{(n)}$ takes $O(IJ^3+N|\Omega|J^N)$.
		  Finally, updating all $\mathbf{A}^{(n)}$ takes $O(NIJ^3+N^2|\Omega|J^N)$.
		  According to \eqref{eq:reconfull}, reconstruction (line 4 in Algorithm~\ref{alg:partial}) takes $O(N|\Omega|J^N)$.
		  Thus, the time complexity of \MOP is $O(NIJ^3+N^2|\Omega|J^N)$.
\end{proof}

\begin{theorem}[Memory complexity of \MOP] \label{lemma:memory_MOP}
	The memory complexity of \MOP is $O(TJ^2)$.
\end{theorem}
\begin{proof} \label{proof:memory_MOP}
	The intermediate data of \MOP consist of
	two vectors $\delta_{\alpha}^{(n)}$ and $\vect{c}_{i_n :}^{(n)}$ ($\in\mathbb{R}^{J}$)
	, and two matrices $\mat{B}_{i_n}^{(n)}$ and $[\mat{B}_{i_n}^{(n)}+\lambda \mathbf{I}_{J_n}]^{-1}$ ($\in\mathbb{R}^{J \times J}$).
	Memory spaces for those variables are released after updating the $i_n$th row of $\mathbf{A}^{(n)}$.
	Thus, they are not accumulated during the iterations.
	Since each thread has their own intermediate data,
 	the total memory complexity of \MOP is $O(TJ^2)$.
\end{proof}

\begin{theorem}[Time complexity of \TOP] \label{lemma:time_TOP}	
	The time complexity of \TOP is $O(NIJ^3+N|\Omega|J^N)$.
\end{theorem}
\begin{proof}
		In Algorithm~\ref{alg:proposed}, computing $\delta$ (line 12) takes $O(N|\Omega|J^N)$ by the caching method. Precomputing and updating $Pres$ (lines 2-4 and 17-19) also take $O(N|\Omega|J^N)$. Since all the other parts of \TOP are equal to those of \method, the time complexity of \TOP is $O(NIJ^3+N|\Omega|J^N)$.
\end{proof}
\begin{theorem}[Memory complexity of \TOP] \label{lemma:memory_TOP}
	The memory complexity of \TOP is $O(|\Omega|J^N)$.
\end{theorem}
\begin{proof}
	The cache table $Pres$ requires $O(|\Omega|J^N)$ memory space, which is much larger than that of other intermediate data (see Theorem~\ref{lemma:memory_MOP}).  Thus, the memory complexity of \TOP is $O(|\Omega|J^N)$.
\end{proof}

\begin{theorem}[Time complexity of \APP] \label{lemma:time_APP}	
	The time complexity of \APP is $O(NIJ^3+N^2|\Omega||\T{G}|)$.
\end{theorem}
\begin{proof}
	Refer to the supplementary material~\cite{supple}.
\end{proof}
\begin{theorem}[Memory complexity of \APP] \label{lemma:memory_APP}
The memory complexity of \APP is $O(J^N)$.
\end{theorem}
\begin{proof}
	Refer to the supplementary material~\cite{supple}.
\end{proof}

	\section{Experiments}
	\label{sec:experiment}

We present experimental results to answer the following questions.

\begin{enumerate}
	\item{\textbf{Data Scalability (Section~\ref{sec:exp:data_scalability}).}
		How well do \method and competitors scale up with respect to the following aspects of a given tensor: 1) the order, 2) the dimensionality, 3) the number of observable entries, and 4) the rank?
	}
	\item{\textbf{Effectiveness of \TOP and \APP  (Section~\ref{sec:exp:optimization}).}
		How successfully do \TOP and \APP suggest the trade-offs between time-memory and time-accuracy, respectively?
	}
	\item{\textbf{Effectiveness of Parallelization (Section~\ref{sec:exp:thread_scalability}).}
		How well does \method scale up with respect to the number of threads used for parallelization?
		How much does the dynamic scheduling accelerate the update process?
	}
	\item{\textbf{Real-World Accuracy (Section~\ref{sec:exp:accuracy}).}
			How accurately do \method and other methods factorize real-world tensors and predict their missing entries?
	}
\end{enumerate}

We describe the datasets and experimental settings in Section~\ref{sec:exp:settings},
and answer the questions in \Cref{sec:exp:data_scalability,sec:exp:optimization,sec:exp:thread_scalability,sec:exp:accuracy}.

\begin{savenotes}
	\begin{table}[htbp!]
		\small
		\centering
		\caption{Summary of real-world and synthetic tensors used for experiments. M: million, K: thousand.}
		\centering
		\begin{tabular}{ r | r | r | r | r}
			\toprule
			\textbf{Name} & \textbf{Order} & \textbf{Dimensionality} & \textbf{$|\Omega|$} & \textbf{Rank} \\
			\midrule
			Yahoo-music & 4 & (1M, 625K, 133, 24) &  252M & 10 \\
			MovieLens & 4 & (138K, 27K, 21, 24)  & 20M & 10\\
			Video (Wave) & 4 & (112,160,3,32) & 160K & 3 \\
			Image (Lena) & 3 & (256,256,3) & 20K & 3 \\
			Synthetic &  3$\thicksim$10 & 100$\thicksim$10M & $\thicksim$100M & 3$\thicksim$11\\
		\bottomrule
		\end{tabular}	
		\label{tab:dataset}
	\end{table}
\end{savenotes}

\subsection{Experimental Settings}
\label{sec:exp:settings}
\subsubsection{Datasets}
\label{sec:datasets}
We use both real-world and synthetic tensors to evaluate \method and competitors.
Table~\ref{tab:dataset} summarizes the tensors we used in experiments, which are available at \textbf{\url{https://datalab.snu.ac.kr/ptucker/}}.
For real-world tensors, we use Yahoo-music\footnote{\url{https://webscope.sandbox.yahoo.com/catalog.php?datatype=r}}, MovieLens\footnote{\url{https://grouplens.org/datasets/movielens/}}, Sea-wave video, and `Lena' image.
Yahoo-music is music rating data which consist of (user, music, year-month, hour, rating).
MovieLens is movie rating data which consist of (user, movie, year, hour, rating).
Sea-wave video and `Lena' image are 10\%-sampled tensors from original data.
Note that we normalize all values of real-world tensors to numbers between 0 to 1. We also use 90\% of observed entries as training data and the rest of them as test data for measuring the accuracy of \method and competitors.
For synthetic tensors, we create random tensors, which we describe in Section~\ref{sec:exp:data_scalability}.

\subsubsection{Competitors}
\label{sec:competitors}
We compare \method and its variants with three state-of-the-art Tucker factorization (TF) methods. Descriptions of all methods are given as follows:
\begin{itemize}
	\item{\textbf{\method (default):}
		the proposed method which minimizes intermediate data by a row-wise update rule, used by default throughout all experiments.
	}
	\item{\textbf{\TOP:}
		the time-optimized variant of \method, which caches intermediate multiplications to update factor matrices efficiently.
	}
	\item{\textbf{\APP:}
		the time-optimized variant of \method, which shows a trade-off between time and accuracy by truncating ``noisy" entries of a core tensor.
	}
	\item{\textbf{\wopt~\cite{filipovic2015tucker}:}
		the accuracy-focused TF method utilizing a nonlinear conjugate gradient algorithm for updating factor matrices and a core tensor.
	}
	\item{\textbf{\CSF~\cite{smith2017tucker}:}
		the speed-focused TF algorithm which accelerates a tensor-times-matrix chain (TTMc) by a compressed sparse fiber (CSF) structure.
	}		
	\item{\textbf{\SHOT~\cite{Oh:2017:SHOT}:}
		the TF method designed for large-scale tensors, which avoids \textit{intermediate data explosion}~\cite{kang2012gigatensor} by \textit{on-the-fly} computation.
	}
\end{itemize}
Notice that other Tucker methods (e.g.,~\cite{kaya, Liu2013}) are excluded since they present similar or limited scalability compared to that of competitors mentioned above.


\begin{figure*}[t!]
	\centering
	\vspace{-3mm}
	\includegraphics[width=0.5\linewidth] {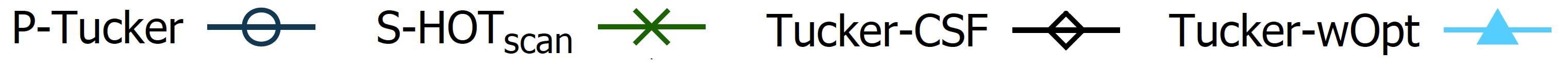} \\
	\hspace{-2mm}
	\begin{subfigure}[t]{0.23\textwidth}
		\includegraphics[width=4.6cm]{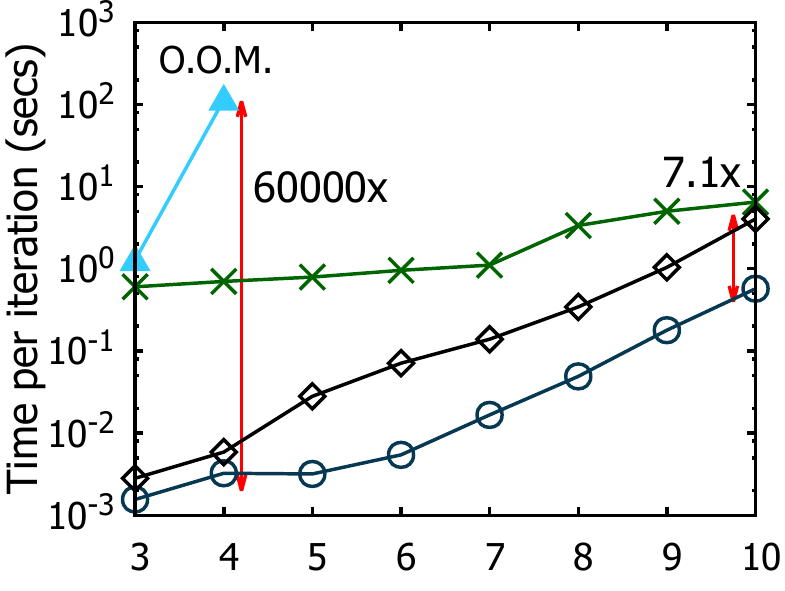}
		\captionsetup{justification=centering}
		\caption{Tensor order.}
		\label{fig:order:full}
	\end{subfigure}
	\hspace{3mm}
	\begin{subfigure}[t]{0.23\textwidth}
		\includegraphics[width=4.6cm]{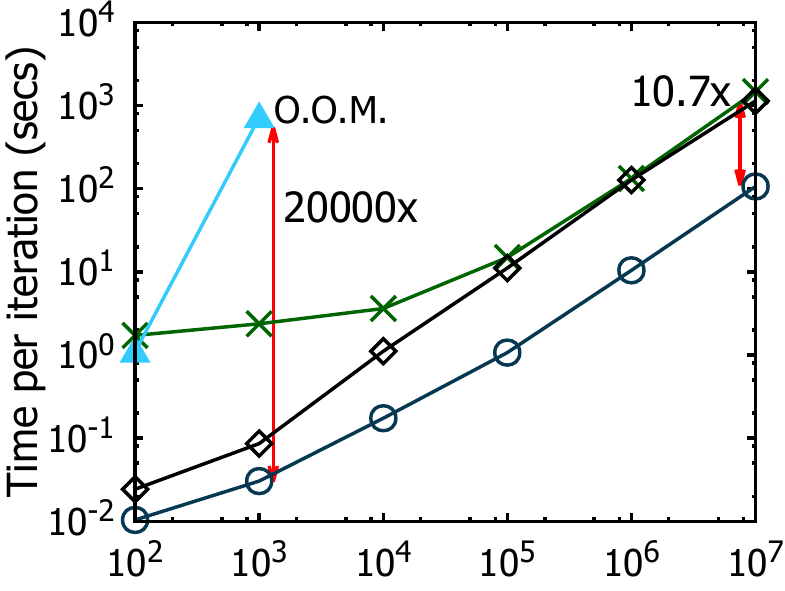}
		\captionsetup{justification=centering}
		\caption{Tensor dimensionality.}
		\label{fig:scalability:dimensionality}
	\end{subfigure}
	\hspace{3mm}
	\begin{subfigure}[t]{0.24\textwidth}
		\centering
		\includegraphics[width=4.6cm]{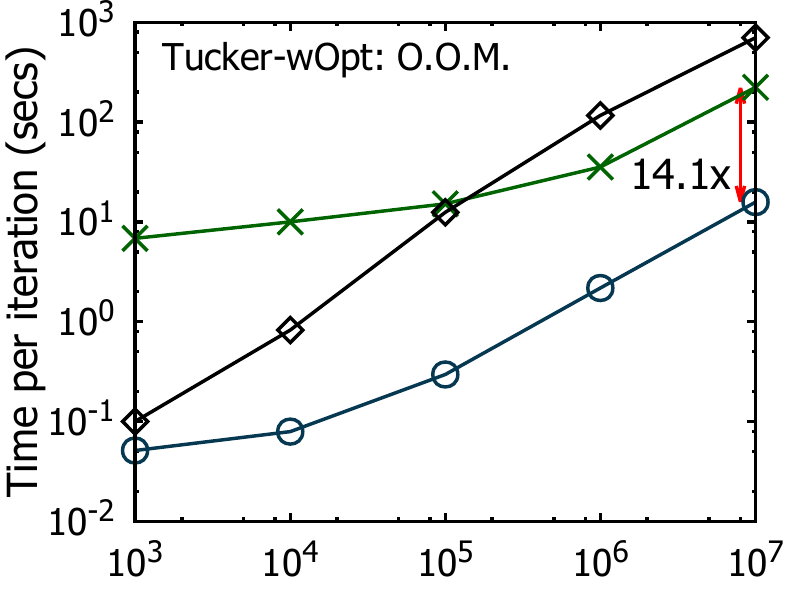}
		\captionsetup{justification=centering}
		\caption{Number of observable entries.}
		\label{fig:scalability:nonzeros}
	\end{subfigure}
	\hspace{1mm}
	\begin{subfigure}[t]{0.23\textwidth}
		\centering
		\includegraphics[width=4.6cm]{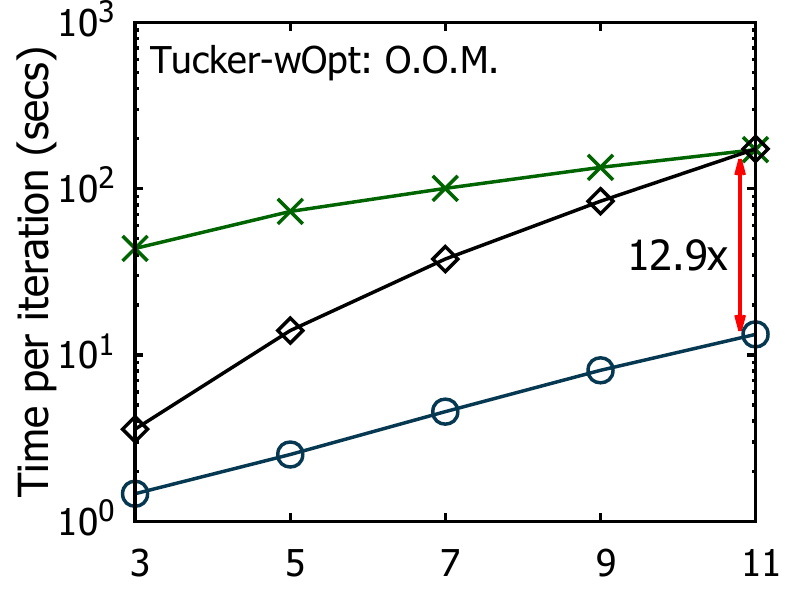}
		\captionsetup{justification=centering}
		\caption{Tensor rank.}
		\label{fig:scalability:rank}
	\end{subfigure}
	\caption{
		The scalability of \method and competitors for large-scale synthetic tensors.
		O.O.M.: out of memory.
		\method exhibits 7.1-14.1x speed up compared to the state-of-the-art with respect to all aspects.
		Notice that \wopt presents O.O.M. in most cases due to their limited scalability, and \method indicates the default memory-optimized version, not \TOP or \APP.
	}
	\label{fig:scalability}
\end{figure*}

\subsubsection{Environment}
\method is implemented in C with \omp and \arma libraries utilized for parallelization and linear algebra operations.
From a practical viewpoint, \method does not automatically choose which optimizations to be used. Hence, users ought to select a method from \method and its variations in advance.
For competitors, we use the original implementations provided by the authors (\SHOT\footnote{\url{https://github.com/jinohoh/WSDM17_shot}}, \CSF\footnote{\url{https://github.com/ShadenSmith/splatt}}, and \wopt\footnote{\url{http://www.lair.irb.hr/ikopriva/Data/PhD_Students/mfilipovic/}}).
We run experiments on a single machine with 20 cores/20 threads, equipped with an Intel Xeon E5-2630 v4 2.2GHz CPU and 512GB RAM.
The default values for \method parameters $\lambda$ and $T$ are set to 0.01 and 20, respectively; for \APP, the truncation rate per iteration is set to 0.2; for \CSF, we set the number of CSF allocations to 1 and choose a LAPACK SVD routine. We set the maximum running time per iteration to 2 hours and the maximum number of iterations to 20. In reporting running times, we use average elapsed time per iteration instead of total running time in order to confirm the theoretical complexities (see Table~\ref{table::analysis}), which are analyzed per iteration.

\subsection{Data Scalability}
\label{sec:exp:data_scalability}

We evaluate the data scalability of \method and other methods using both synthetic and real-world tensors.

\subsubsection{Synthetic Data}

We generate random tensors of size $I_1\!=\!I_2\!=\!...\!=\!I_N$ with real-valued entries between 0 and 1,
varying the following aspects: tensor order, tensor dimensionality, the number of observable entries, and tensor rank.
We assume that the core tensor $\T{G}$ is of size $J_1\!=\!J_2\!=\!...\!=\!J_N$.

\textbf{Order.}
We increase the order $N$ of an input tensor from 3 to 10, while fixing $I_n\!=\!10^2$, $|\Omega|\!=\!10^3$, and $J_n\!=\!3$.
As shown in Figure~\ref{fig:order:full}, \method exhibits the fastest running time with respect to the order.
Although \SHOT and \CSF can decompose up to the highest-order tensor, they run 11$\times$ and 7.1$\times$ slower than \method, respectively.
\wopt runs 60000$\times$ (when $N=4$) slower than \method and shows O.O.M. (out of memory error) when $N\geq5$. The enormous speed-gap between \method and \wopt is explained by their time complexities. The speed of \wopt mainly depends on the dimensionality term $I^N$, while \method relies on the rank term $J^N$ where $I >> J$.

\textbf{Dimensionality.}
We increase the dimensionality $I_n$ of an input tensor from $10^2$ to $10^7$,
while setting $N\!=\!3$, $|\Omega|\!=\!10 \times I_n$, and $J_n\!=\!10$.
As shown in Figure~\ref{fig:scalability:dimensionality},
\method consistently runs faster than other methods across all dimensionality.
\wopt runs $20000 \times$ (when $I_n=10^3$) slower than \method and presents O.O.M. when $I_n\geq 10^4$. The speed-gap between \method and \wopt is also described in a similar way to that of the order case.
Though \SHOT and \CSF scale up to the largest tensor as well, they run  13.8$\times$ and 10.7$\times$ slower than \method, respectively.

\textbf{Number of Observable Entries.}
We increase the number $|\Omega|$ of observable entries from $10^3$ to $10^7$, while fixing $N\!=\!3$, $I_n\!=\!10^7$, and $J_n\!=\!10$.
As shown in Figure~\ref{fig:scalability:nonzeros}, \method, \SHOT, and \CSF scale up to the largest tensor,
while \wopt shows O.O.M. for all tensors.
\method presents the fastest factorization speed across all $|\Omega|$ and runs 14.1$\times$ and 44.3$\times$ faster than \SHOT and \CSF on the largest tensor with $|\Omega|=10^7$, respectively. Note that \method scales near linearly with respect to the number of observable entries.

\textbf{Rank.}
We increase the rank $J_n$ from 3 to 11 with an increment of 2, while fixing $N\!=\!3$, $I_n\!=\!10^6$, and $|\Omega|\!=\!10^7$.
As shown in Figure~\ref{fig:scalability:rank}, \method, \SHOT, and \CSF successfully factorize input tensors for all ranks.
\method is the fastest in all cases; in particular, \method runs 12.9$\times$ and 13.0$\times$ faster than \SHOT and \CSF when $J_n=11$, respectively.
\wopt causes O.O.M. errors for all ranks.

\begin{figure}[t!]
	\centering
	\includegraphics[width=0.49\textwidth]{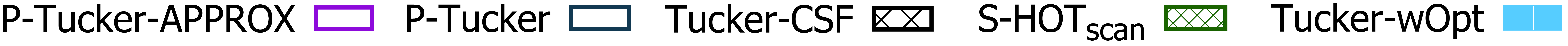} \\
	\vspace{1mm}
	\begin{subfigure}[t]{0.24\textwidth}
		\centering
		\includegraphics[width=4.5cm]{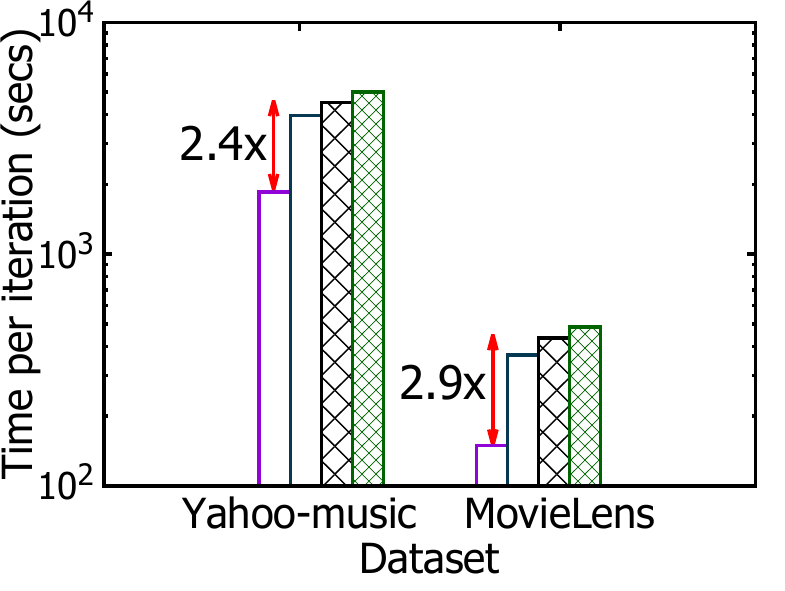}
		\captionsetup{justification=centering}
		\label{fig:real_world_large}
	\end{subfigure}
	\begin{subfigure}[t]{0.24\textwidth}
		\centering
		\includegraphics[width=4.5cm]{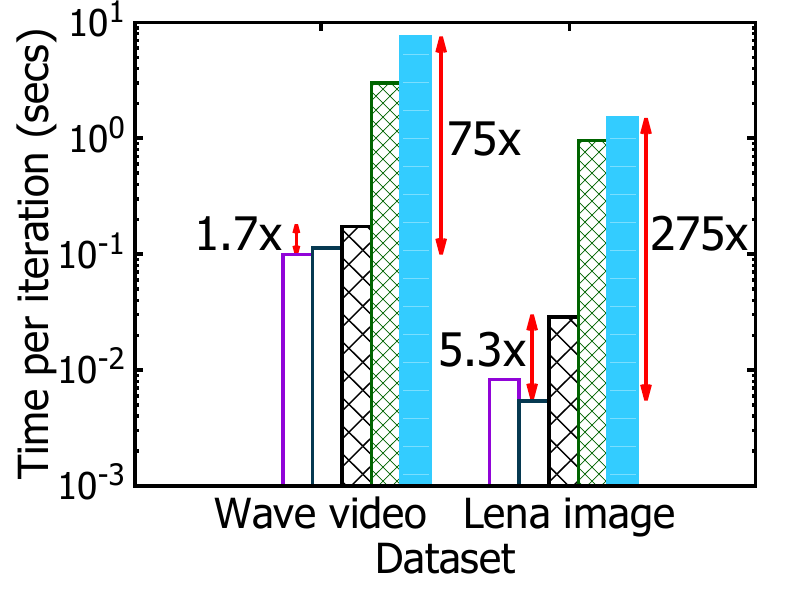}
		\captionsetup{justification=centering}
		\label{fig:real_world_small}
	\end{subfigure}
	\vspace{-3mm}
	\caption{
		The scalability of \method and competitors on real-world tensors. \method and \APP show the fastest running time across all datasets. An empty bar indicates that the corresponding method shows O.O.M. while factorizing the dataset.
	}
	\label{fig:real_world_time}
\end{figure}

\subsubsection{Real-world Data}
We measure the average running time per iteration of \method and other methods on the real-world datasets introduced in Section~\ref{sec:datasets}. Due to the large scale of real-world tensors, \wopt shows O.O.M. for two of them, which are set to blanks as shown in Figure~\ref{fig:real_world_time}.
Notice that \method and \APP succeed in decomposing the large-scale real-world tensors and run $1.7-275\times$ faster than competitors.

\subsection{\TOP and \APP}
\label{sec:exp:optimization}
To investigate the effectiveness of \TOP, we vary the tensor order $N$ from $6$ to $10$, while fixing $I_n\!=\!10^2$, $|\Omega|\!=\!10^3$, and $J_n\!=\!3$.
Figure~\ref{fig:order} shows the running time and memory usage of \method and \TOP.
\method uses $29.5\times$ less memory than \TOP for the largest order $N=10$.
However, \TOP runs up to $1.7\times$ faster than \method, where the gap between the running times grows as tensor order $N$ grows since running times of \TOP and \MOP are mainly proportional to $N$ and $N^2$, respectively.

In the case of \APP, we measure per-iteration time and full running time until convergence.
Figures~\ref{fig:APP_time} and ~\ref{fig:APP_recon} illustrate the effectiveness of \APP for the MovieLens dataset ($J_n\!=\!5$). \APP gets faster than \method when iteration $\geq 3$ and  converges $1.7 \times$ earlier than \method. Moreover, the reconstruction error of \APP is almost the same as that of \method.
Note that one iteration corresponds to lines 3-6 in Algorithm~\ref{alg:partial}.

\begin{figure}[t!]
	\centering
	\vspace{-3mm}
	\hspace{-4mm}
	\begin{subfigure}[t]{0.24\textwidth}
		\centering
		\includegraphics[width=4.6cm]{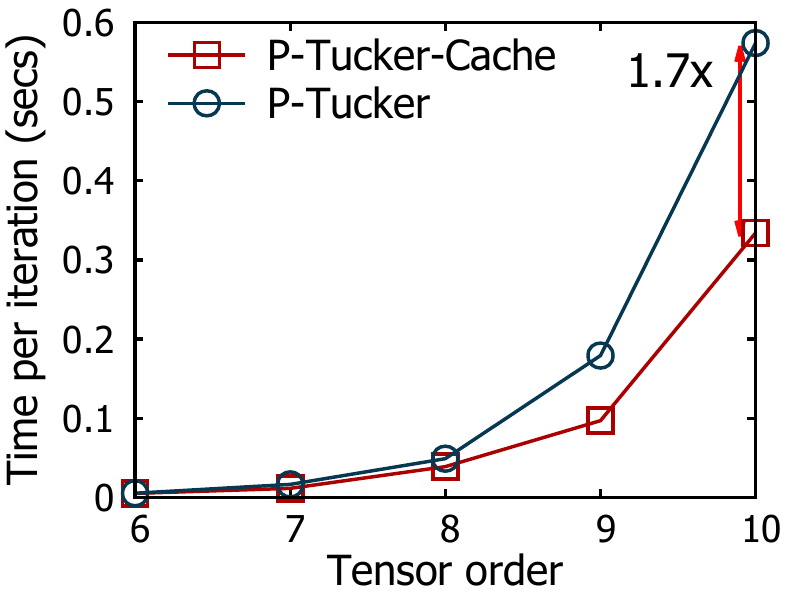}
		\captionsetup{justification=centering}
		\caption{Running time of \method \\ and \TOP. }
		\label{fig:order:time}
	\end{subfigure}
	\hspace{2mm}
	\begin{subfigure}[t]{0.24\textwidth}
		\centering
		\includegraphics[width=4.6cm]{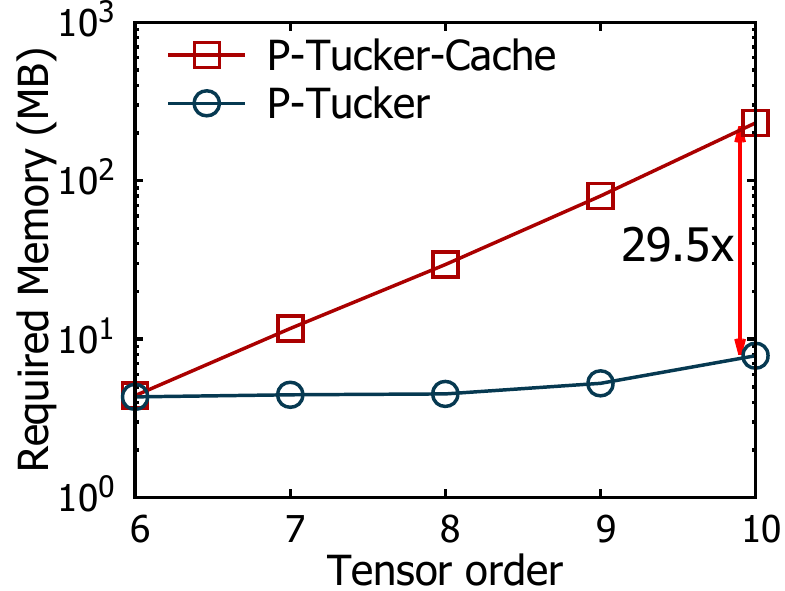}
		\captionsetup{justification=centering}
		\caption{Memory usage of \method \\ and \TOP. }
		\label{fig:order:memory}
	\end{subfigure}
	\caption{
		Comparison results of \method and \TOP. \TOP runs up to 1.7$\times$ faster than \MOP for higher-order tensors, while \MOP decomposes the highest-order tensor with 29.5$\times$ less memory than \TOP.
	}
	\label{fig:order}
\end{figure}

\begin{figure}[t!]
	\centering
	\hspace{-6mm}
	\begin{subfigure}[t]{0.23\textwidth}
		\includegraphics[width=4.6cm]{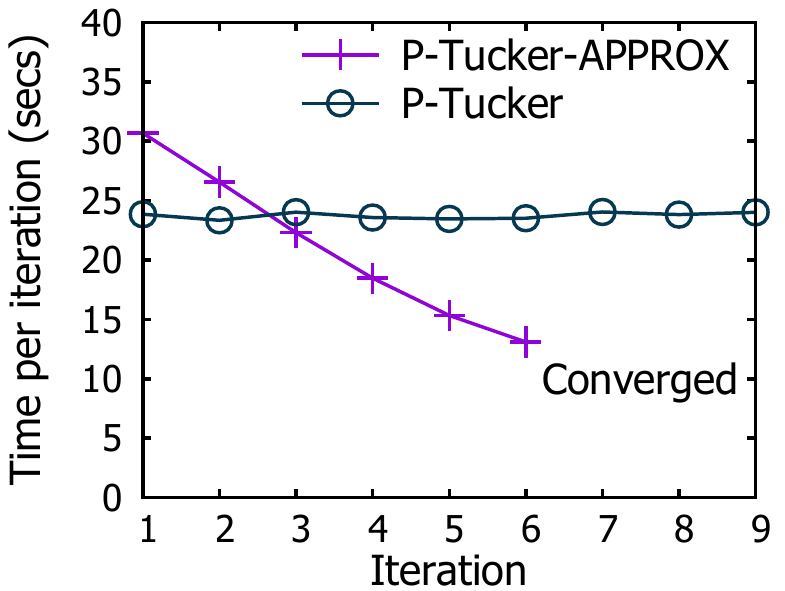}
		\captionsetup{justification=centering}
		\caption{Per-iteration running time of \method \\ and \APP. }
		\label{fig:APP_time}
	\end{subfigure}
	\hspace{3mm}
	\begin{subfigure}[t]{0.23\textwidth}
		\includegraphics[width=4.6cm]{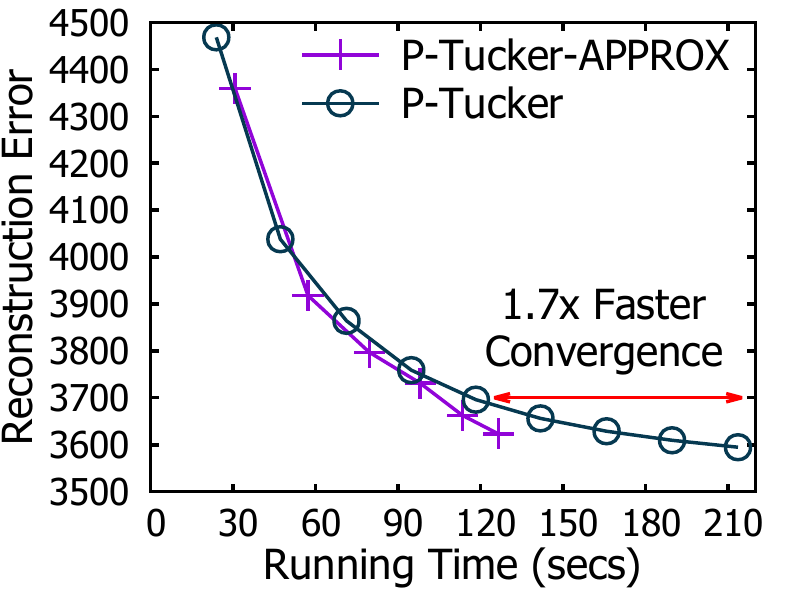}
		\captionsetup{justification=centering}
		\caption{Accuracy of \method and \APP until convergence. }
		\label{fig:APP_recon}
	\end{subfigure}
	\caption{Comparison results of \method and \APP. \APP gets faster at every iteration and eventually runs quicker than \method (when iteration~$\geq 3$). Furthermore, \APP converges $1.7 \times$ faster than \method with almost the same accuracy.}
	\label{fig:APP_exp}
\end{figure}

\subsection{Effectiveness of Parallelization}
\label{sec:exp:thread_scalability}
We measure the speed-ups ($Time_{1}/Time_{T}$ where $Time_{T}$ is the running time using $T$ threads) and memory requirements of \method by increasing the number of threads from 1 to 20, while fixing $N\!=\!3$, $I_n\!=\!10^6$, and $|\Omega|\!=\!10^7$.
Figure~\ref{fig:thread_scalability} shows near-linear speed up and memory requirements of \method regarding the number of threads. The linear speed-up implies that our parallelization works successfully, and the linearity of memory usage demonstrates that our theoretical memory complexity of \method matches the empirical result well. In addition, in order to verify the speed-up of dynamic scheduling, we compare \method with a naive parallelization which does not consider workload distributions. For the MovieLens dataset ($J_n = 10$), the running time of \method (367.5s) is $1.5 \times$ faster than that of the naive approach (552.7s), which demonstrates the effectiveness of dynamic scheduling.

\begin{figure}[t!]
	\centering
	\vspace{-3mm}
	\hspace{-2mm}
	\begin{subfigure}[t]{0.23\textwidth}
		\includegraphics[width=4.6cm]{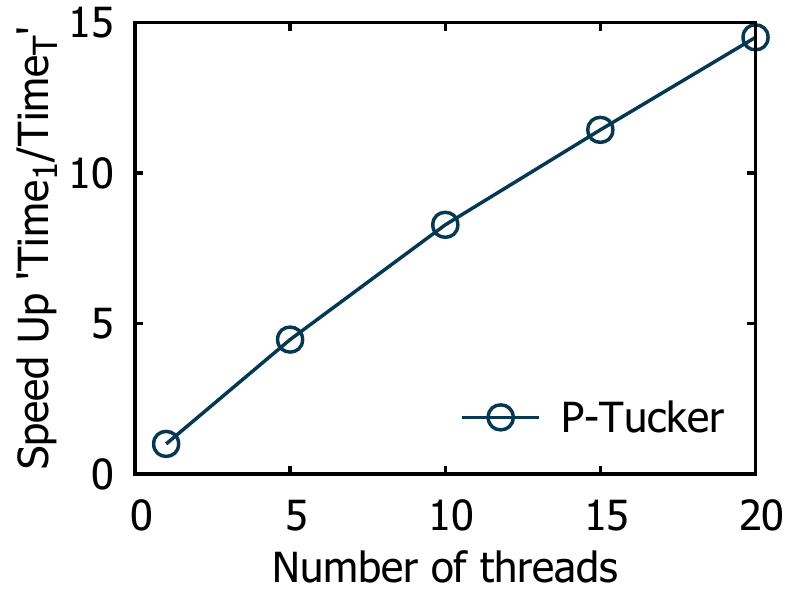}
		\label{fig:thread_time}
	\end{subfigure}
	\hspace{3mm}
	\begin{subfigure}[t]{0.23\textwidth}
		\includegraphics[width=4.6cm]{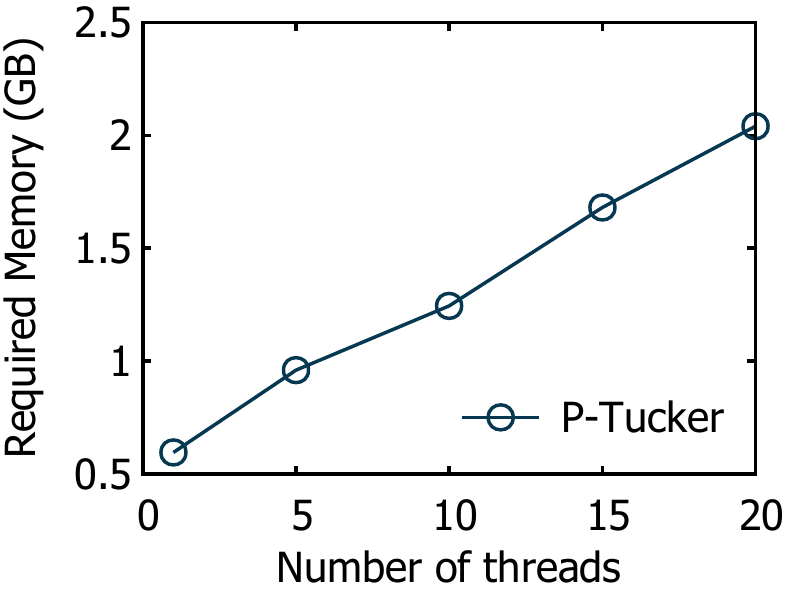}
		\label{fig:thread_memory}
	\end{subfigure}
	\vspace{-2mm}
	\caption{The parallelization scalability of \method. Notice that the speed of \method increases linearly in terms of the number of threads, and the memory requirements of \method also scale near linearly with regard to the number of threads.}
	\label{fig:thread_scalability}
\end{figure}

\subsection{Real-World Accuracy }
\label{sec:exp:accuracy}
We evaluate the accuracy of \method and other methods on the real-world tensors.
The evaluation metrics are reconstruction error and test root mean square error (RMSE); the former describes how precisely a method factorizes a given tensor, and the latter indicates how accurately a method estimates missing entries of a tensor, which is widely used by recommender systems.
As shown in Figure~\ref{fig:real_world_accuracy},
\method factorizes the tensors with 1.4-4.8$\times$ less reconstruction error and predicts missing entries of given tensors with 1.4-4.3$\times$ less test RMSE compared to the state-of-the-art.
In particular, \method exhibits 1.4-2.6$\times$ higher accuracy than that of \wopt, which also focuses on observed entries during factorizations.
In Figure~\ref{fig:real_world_accuracy}, we present \SHOT and \CSF with the same bar since they have similar accuracy, and the methods have low accuracies as they try to estimate missing entries as zeros.
An omitted bar indicates that the corresponding method shows O.O.M. while decomposing the dataset.

\begin{figure}[h!]
	\centering
	\vspace{3mm}
	\includegraphics[width=0.5\textwidth]{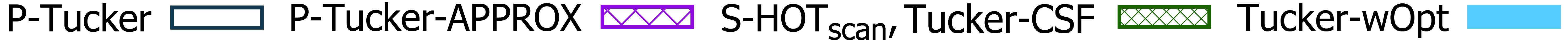} \\
	\hspace{-1mm}
	\begin{subfigure}[t]{0.23\textwidth}
		\includegraphics[width=4.6cm]{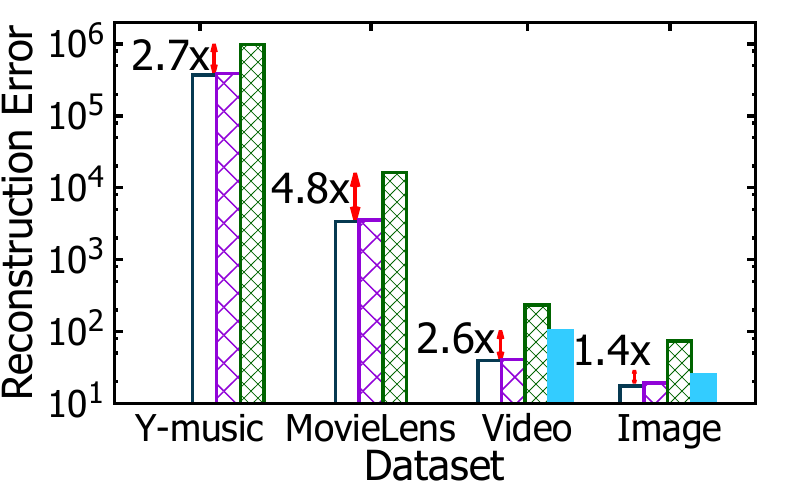}
		\label{fig:real_world_recon}
	\end{subfigure}
	\hspace{1mm}
	\begin{subfigure}[t]{0.23\textwidth}
		\includegraphics[width=4.6cm]{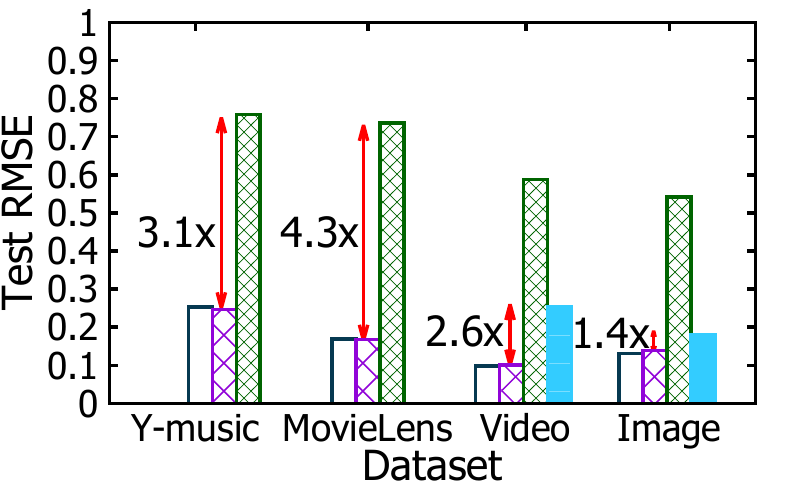}
		\label{fig:real_world_rmse}
	\end{subfigure}
	\vspace{-3mm}
	\caption{The accuracy of \method and competitors on the real-world tensors. \method achieves 1.4-4.8$\times$ higher accuracy compared to that of existing methods. An empty bar indicates that the corresponding method shows O.O.M. while factorizing the dataset.}
	\label{fig:real_world_accuracy}
\end{figure}




	\section{Discovery}
	\label{sec:discovery}
	We present discoveries on the latest MovieLens dataset introduced in Section~\ref{sec:exp:settings}. Existing methods cannot detect meaningful concepts or relations owing to their limited scalability or low accuracy.
For instance, \SHOT and \CSF produce factor matrices mostly filled with zeros, which trigger highly inaccurate clustering. In contrast, \method successfully reveals the hidden concepts and relations such as a `Thriller' concept, and a relation between a `Drama' concept and hours (see Tables~\ref{tab:concept} and~\ref{tab:relation}).

\textbf{Concept Discovery.}
Our intuition for concept discovery is that each row of factor matrices represents latent features of the row.
Thus, we can apply K-means clustering algorithm~\cite{kmeans} on factor matrices to discover hidden concepts.
In the case of movie-associated factor matrix, each row represents a latent feature of a movie.
Therefore, by analyzing the clustered rows,  \method excavates diverse movie genres, such as `Thriller', `Comedy', and `Drama', and all the movies belonging to those genres are closely related (see Table~\ref{tab:concept}).

\textbf{Relation Discovery.}
Core tensor $\T{G}$ plays an important role in discovering relations.
An entry $(j_1,...,j_N)$ of $\T{G}$ is associated with the $j_n$th column of $\mathbf{A^{(n)}}$, and it implies that those columns are related to each other with a strength $\T{G}_{(j_1,...,j_N)}$. Hence,
examining large values in $\T{G}$ gives us clues to find strong relations in a given tensor.
For instance, \method succeeds in revealing relations between year and hour attributes such as $(2015, 2 pm)$ by investigating the $top-3$ largest value of a core tensor. In a similar way, \method finds strong relations between movie, year, and hour attributes, as summarized in Table~\ref{tab:relation}.

	\section{Related Work}
	\label{sec:related_work}
	
We review related works on CP and Tucker
factorizations, and applications of Tucker decomposition.

\textbf{CP Decomposition (CPD).}
Many algorithms have been developed for scalable CPD.
GigaTensor~\cite{kang2012gigatensor} is the first distributed CP method running on the MapReduce framework.
Park et al.~\cite{DBTF_ParkICDE17} propose a distributed algorithm, DBTF, for fast and scalable Boolean CPD.
In \cite{parcube}, Papalexakis et al. present a sampling-based, parallelizable method named ParCube for sparse CPD.
AdaTM~\cite{li2017model} is an adaptive tensor memoization algorithm for CPD of sparse tensors, which automatically tunes algorithm parameters.
Kaya and U{\c{c}}ar~\cite{kaya2015scalable} propose distributed memory CPD methods based on hypergraph partitioning of sparse tensors.
Those algorithms are based on the ALS similarly to the conventional Tucker-ALS.


Since the above CP methods predict missing entries as
zeros, tensor completion algorithms using CPD have gained increasing attention in recent years.
Tomasi et al.~\cite{tomasi2005parafac} and Acar et al.~\cite{Acar2011} first address CPD models for tensor completion problems.
Karlsson et al.~\cite{KARLSSON2016222} discuss parallel formulations of ALS and CCD++ for tensor completion in the CP format.
Smith et al.~\cite{smith2016exploration} explore three optimization
algorithms for high performance, parallel tensor completion:
alternating least squares (ALS), stochastic gradient descent (SGD),
and coordinate descent (CCD++).
For distributed platforms, Shin et al.~\cite{ShinK17} propose CDTF and SALS, which are ALS-based CPD methods for partially observable tensors; Yang et al.~\cite{Yang2017} also offer SGD-based formulations for sparse tensors.
Note that \cite{ShinK17} and \cite{smith2016exploration} offer a  row-wise parallelization for CPD as \method does for Tucker decomposition.

\begin{table}[t!]
	\footnotesize
	\centering
	\caption{Concept discoveries on the MovieLens dataset ($J=8,K=100$). Three notable movie concepts are found by \method. }
	\centering
	\begin{tabular}{ c | c | c }
		\toprule
		\textbf{Concept} & \textbf{Index} & \textbf{Attributes} \\
		\midrule
		& 15535	& Inception, 2010, $Action|Crime|Sci-Fi$\\
		\underline{\textbf{C1: Thriller}} & 4880 & Vanilla Sky, 2001, $Mystery|Romance$  \\
		& 24694 & The Imitation Game, 2014, $Drama|Thriller$ \\ \hline
		& 6373 & Bruce Almighty, 2003, $Drama|Fantasy$\\
		\underline{\textbf{C2: Comedy}} & 16680 & Home Alone 4, 2002, $Children|Comedy$\\
		& 12811 & Mamma Mia!, 2008, $Musical|Romance$\\ \hline
		& 19822 & Life of Pi, 2012, $Adventure|Drama|IMAX$\\
		\underline{\textbf{C3: Drama}} & 11873 & Once, 2006, $Drama|Musical|Romance$ \\
		& 214 & Before Sunrise, 1995, $Drama|Romance$ \\
		\bottomrule
	\end{tabular}	
	\label{tab:concept}
	\vspace{2mm}
\end{table}

\begin{table}[t!]
	\footnotesize
	\centering
	\caption{Relation discoveries on the MovieLens dataset. Three notable relations between movie, year, and hour are found by \method. }
	\centering
	\begin{tabular}{ c | c | c }
		\toprule
		\textbf{Relations} & \textbf{$\T{G}$ Value} & \textbf{Details} \\
		\midrule
		\underline{\textbf{R1:}} & $1.65 \times 10^6$ & Most preferred hours for drama genre  \\
		\underline{\textbf{Drama-Hour}} & & 8 am, 4 pm, 1 am, 9 pm, and 6 pm \\ \hline
		\underline{\textbf{R2: }} & $1.29 \times 10^6$ & Comedy genre is favored in this period \\
		\underline{\textbf{Comedy-Year}} & &  (1997, 1998, 1999), (2005, 2006, 2007)\\ \hline
		\underline{\textbf{R3: }} & $2.29 \times 10^6$& Most preferred hour for watching movies\\
		\underline{\textbf{Year-Hour}} & & (2015, 2 pm), (2014, 0 am), (2013, 9 pm) \\
		\bottomrule
	\end{tabular}	
\vspace{2mm}
	\label{tab:relation}
\end{table}

\textbf{Tucker Factorization (TF).}
Several algorithms have been developed for TF.
\cite{tucker1966some} presents an early work on TF, which is known as HOSVD.
De Lathauwer et al.~\cite{DBLP:journals/siammax/LathauwerMV00a} propose Tucker-ALS, described in Algorithm~\ref{alg:ALS_FULL}.
As the size of real-world tensors increases rapidly,
there has been a growing need for scalable TF methods.
One major challenge is the ``intermediate data explosion'' problem~\cite{kang2012gigatensor}.
%
MET (Memory Efficient Tucker)~\cite{MET} tackles this challenge
by adaptively ordering computations and performing them in a piecemeal manner.
HaTen2~\cite{DBLP:journals/vldb/JeonPFSK16} reduces intermediate data
by reordering computations and exploiting the sparsity of real-world tensors in MapReduce.
However, both MET and HaTen2 suffer from a limitation called M-bottleneck~\cite{Oh:2017:SHOT}
that arises from explicit materialization of intermediate data.
S-HOT~\cite{Oh:2017:SHOT} avoids M-bottleneck by employing on-the-fly computation.
Kaya and U{\c{c}}ar~\cite{kaya} discuss a shared and distributed memory parallelization of an ALS-based TF for sparse tensors.
\cite{ChakaravarthyCJ17} proposes optimizations of HOOI for dense tensors on distributed systems.
The above methods depend on SVD for updating factor matrices, while \method utilizes a row-wise update rule.

There are also various accuracy-focused TF methods including \wopt~\cite{filipovic2015tucker}.
Yang et al.~\cite{DBLP:journals/tsp/YangFLZ16} propose another TF method
that automatically finds a concise Tucker representation of a tensor via an iterative reweighted algorithm.
Liu et al.~\cite{Liu2013} define the trace norm of a tensor, and present three convex optimization algorithms for low-rank tensor completion.
%
Liu et al.~\cite{DBLP:conf/nips/LiuSFCC14} propose a core tensor Schatten 1-norm minimization method with a rank-increasing scheme for tensor factorization and completion.
Note that these algorithms have limited scalability compared to \method since they are not fully optimized with respect to time and memory.


%


\textbf{Applications of Tucker Factorization.}
Tucker factorization (TF) has been used for various applications.
Sun et al.~\cite{Websearch} apply a 3-way TF to a tensor consisting of (users, queries, Web pages) to personalize Web search.
Bro et al.~\cite{bro1999fast} use TF for speeding up CPD by compressing a tensor.
Sun et al.~\cite{sun2009multivis} propose a framework for content-based network analysis and visualization
which employs a biased sampling-based TF method.
TF is also used for analyzing trends in the blogosphere~\cite{DBLP:conf/cikm/ChiTT06}.

	\section{Conclusion}
	\label{sec:conclusion}
	We propose \method, a scalable Tucker factorization method for sparse tensors.
By using ALS with a row-wise update rule, and with careful distributions of works for parallelization, \method successfully offers time and memory optimized algorithms with theoretical proof and analysis.
\method runs 1.7-14.1$\times$ faster than the state-of-the-art with 1.4-4.8$\times$ less error, and
exhibits near-linear scalability with respect to the number of observable entries and threads.
We discover hidden concepts and relations on the latest MovieLens dataset with \method, which cannot be identified by existing methods due to their limited scalability or low accuracy.
Future works include extending \method to distributed platforms such as Hadoop or Spark, and applying sampling techniques on observable entries to accelerate decompositions, while sacrificing little accuracy.

	\section*{Acknowledgment}
	{
		This work was supported by the National Research Foundation of Korea (NRF) funded by the Ministry of Science, ICT, and Future Planning (NRF-2016M3C4A7952587, PF Class Heterogeneous High-Performance Computer Development). U Kang is the corresponding author.
	}	
	
	\bibliographystyle{IEEEtran}
	\bibliography{bib/myref,bib/ukang}
	
\end{document}